\documentclass[11pt]{amsart}

\usepackage[utf8]{inputenc}
\usepackage[english]{babel}
\usepackage[a4paper,twoside,top=1.2in, bottom=1.1in, left=0.8in, right=0.8in]{geometry}

\usepackage{graphicx}
\usepackage{caption} 
\usepackage{amsmath}
\usepackage{amsthm}
\usepackage{amssymb}
\usepackage{esint} 
\usepackage{mathrsfs} 
\usepackage{xcolor}
\usepackage{array}
\usepackage{hhline}
\usepackage{enumitem} 
\usepackage{comment} 
\usepackage[toc,page]{appendix} 
\usepackage{xparse} 
\usepackage{mathtools}
\usepackage{fancyhdr} 
\usepackage{ifthen} 
\usepackage{forloop} 
\usepackage{xstring}
\usepackage{emptypage} 
\usepackage{setspace}
\usepackage[initials,alphabetic]{amsrefs} 

\usepackage{tikz}
\usetikzlibrary{arrows,shapes,patterns,calc,fadings,decorations.pathreplacing,decorations.markings,decorations.pathmorphing,backgrounds}

\usepackage{hyperref} 
\hypersetup{
	colorlinks = true,
	linkcolor = {blue},
	urlcolor = {red},
	citecolor = {blue}
}

\usepackage[nameinlink,capitalise]{cleveref} 

\newcounter{results}[section] 

\theoremstyle{plain}
\newtheorem{theorem}[results]{Theorem}
\newtheorem{lemma}[results]{Lemma}
\newtheorem{proposition}[results]{Proposition}
\newtheorem{corollary}[results]{Corollary}

\newtheorem*{theorem*}{Theorem}
\newtheorem*{lemma*}{Lemma}
\newtheorem*{proposition*}{Proposition}
\newtheorem*{corollary*}{Corollary}
\newtheorem*{exercise*}{Exercise}
\newtheorem*{fact*}{Fact}

\newtheorem*{theorem3}{Theorem 7.5 for $n+1=3$}

\theoremstyle{remark}
\newtheorem{remark}[results]{Remark}

\newtheorem*{remark*}{Remark}
\newtheorem*{question*}{Question}
\usepackage[many]{tcolorbox}

\newtheorem*{case3}{Case $n+1=3$}
\tcolorboxenvironment{case3}{
  colback=blue!3!white,
  boxrule=0pt,
  boxsep=2pt,
  left=5pt,right=5pt,top=2pt,bottom=2pt,
  oversize=2pt,
  rounded corners,
  before skip=10pt,
  after skip=10pt,
  breakable
}

\theoremstyle{definition}
\newtheorem{definition}[results]{Definition}

\newtheorem*{definition*}{Definition}
\newtheorem*{example*}{Example}

\numberwithin{equation}{section}

\crefname{figure}{Figure}{Figures}

\ifdefined\comma 
        \renewcommand{\comma}{\ensuremath{\, \text{, }}}
\else 
        \newcommand{\comma}{\ensuremath{\, \text{, }}}
\fi

\newcommand{\point}{\ensuremath{\, \text{. }}}

\newcommand{\N}{\ensuremath{\mathbb N}}
\newcommand{\Z}{\ensuremath{\mathbb Z}}
\newcommand{\R}{\ensuremath{\mathbb R}}

\DeclarePairedDelimiter\abs{\lvert}{\rvert} 

\newcommand{\st}{\ensuremath{\ :\ }} 
\newcommand{\eqdef}{\ensuremath{\coloneqq}} 

\renewcommand{\d}{\ensuremath{d}} 
\newcommand{\de}{\ensuremath{\, d}} 

\newcommand{\lapl}{\ensuremath{\Delta}} 

\newcommand{\vf}{\ensuremath{\mathfrak X}} 
\newcommand{\cov}{\ensuremath{\nabla}} 
\DeclareMathOperator{\II}{I\!I} 
\DeclareMathOperator{\Ric}{Ric} 
\newcommand{\dist}{\ensuremath{d}} 

\newcommand{\Haus}{\ensuremath{\mathscr H}} 
\newcommand{\ms}{\ensuremath{\Sigma}} 
\newcommand{\amb}{\ensuremath{M}} 
\DeclareMathOperator{\ind}{ind} 

\DeclareMathOperator{\radius}{radius}
\DeclareMathOperator{\Sing}{Sing}

\newcommand{\conical}{\mathfrak{C}}
\newcommand{\CA}{\ensuremath{\Lambda}}
\newcommand{\BO}{\ensuremath{\Upsilon}}
\newcommand{\BU}{\ensuremath{\Omega}}

\DeclareMathOperator{\conv}{conv}

\newcommand{\PJ}{\normalfont P[J]}

\colorlet{myGray}{gray}
\colorlet{myBlue}{blue}
\colorlet{myBlack}{black}
\colorlet{myBackground}{gray!10}

\allowdisplaybreaks

\title[Estimating the Morse index of FBMH through covering arguments]{Estimating the Morse index of \\ free boundary minimal hypersurfaces \\ through covering arguments}
\author{Santiago Cordero-Misteli and Giada Franz}
\newcommand\printaddress{{
\setlength{\parindent}{17pt}
\footnotesize

\bigskip
\par
{\scshape \noindent Santiago Cordero-Misteli}
\newline SUNY Stony Brook, Department of Mathematics, Stony Brook University, Stony Brook, NY 11794, USA
\newline
\textit{E-mail address:} \texttt{santiago.corderomisteli@stonybrook.edu}
\newline 
\par 
{\scshape \noindent Giada Franz}
\newline MIT, Department of Mathematics, 77 Massachusetts Avenue, Cambridge, MA 02139, USA
  \newline
\textit{E-mail address:} \texttt{gfranz@mit.edu}
\par
}}

\begin{document}
\maketitle

\begin{abstract}
Given a compact Riemannian manifold, of dimension between $3$ and $7$, with boundary, we adapt Song's method in \cite{Song2019} to the free boundary case to show that the Morse index of a free boundary minimal hypersurface grows linearly with the sum of its Betti numbers, where the constant of growth depends on an upper bound on the area of the free boundary minimal hypersurface in question.
\end{abstract}

\thispagestyle{empty}

\section{Introduction}

The goal of this paper is to prove a quantitative bound on the sum of the Betti numbers of a free boundary minimal hypersurface in a compact Riemannian manifold $M^{n+1}$, of dimension $3\le n+1\le 7$, with boundary in terms of the area and the Morse index of the hypersurface.

The question whether, and how, the Morse index of a minimal hypersurface bounds its topology goes back already to the classical theory of complete minimal surfaces in $\R^3$. Indeed, Fischer-Colbrie proved in \cite{FischerColbrie1985} that a complete minimal surface in $\R^3$ with finite Morse index has finite total curvature, which implies restrictions on the topology and conformal structure of the minimal surface by \cite{Osserman1964}.
Nowadays, we know that the index of a minimal surface in $\R^3$ bounds from above a linear combination of its genus and the number of its ends by \cites{ChodoshMaximo2016,ChodoshMaximo2018}.

In the case of a boundaryless compact ambient manifold $(M^{n+1},g)$ with \emph{positive Ricci curvature}, it was conjectured by Marques--Neves--Schoen (see \cite{Neves2014}*{Section 8}, and also \cite{AmbrozioCarlottoSharp2018Index}*{pp.~3} and \cite{Song2019}*{eq.~(1)}) that there exists a constant $C>0$ (depending on the ambient manifold $(M,g)$) such that
\[
b_1(\ms)\le C\ind(\ms), 
\]
for every closed minimal surface $\ms^n\subset M$, where $b_1(\ms)$ is the first Betti number of $\ms$.
This conjecture is still open, but it has been proven under additional assumptions (e.g. pinching assumptions on the curvature) in \cite{AmbrozioCarlottoSharp2018Index} (see also \cites{Ros2006,Savo2010} for previous works, and \cite{GorodskiMendesRadeschi2019,MendesRadeschi2020} for further extensions).

In the case with boundary, the assumption corresponding to ``positive Ricci curvature'' is that the ambient manifold $(M^{n+1},g)$ has \emph{either} positive Ricci curvature and convex boundary \emph{or} nonnegative Ricci curvature and strictly convex boundary.
In this setting, one may formulate the same conjecture as in the closed case by Marques--Neves--Schoen, which has been proven so far only under additional assumptions in \cite{AmbrozioCarlottoSharp2018IndexFBMS} (see also \cite{Sargent2017}).

In the case of a boundaryless compact three-dimensional ambient manifold $(M^3,g)$ with \emph{positive scalar curvature}, only qualitative bounds on the topology in terms of the index are at disposal. Namely, Chodosh--Ketover--Maximo proved in \cite{ChodoshKetoverMaximo2017}*{Theorem 1.3} that the space of minimal surfaces $\ms^2\subset M$ with prescribed index have uniformly bounded genus (and area).
In the analogous setting with boundary, namely three-dimensional ambient manifolds with \emph{either} positive scalar curvature and mean convex boundary \emph{or} nonnegative scalar curvature and strictly mean convex boundary, the same result was proven in  \cite{CarlottoFranz2020}*{Theorem 1.4}.

Observe that, in absence of any curvature assumption, we cannot hope for a bound on the topology of a minimal surface {solely} in terms of its index (see \cite{ChodoshKetoverMaximo2017}*{Example 1.16}). However, given a boundaryless compact ambient manifold $(M^{n+1},g)$ of dimension $3\le n+1\le 7$, Sharp proved in \cite{Sharp2017} that a bound on the index and the area of a minimal hypersurface $\ms^n\subset M$ implies a bound on its genus (see \cite{AmbrozioCarlottoSharp2018Compactness} for the case with boundary). 
This result was improved by Song in \cite{Song2019}*{Theorem~1} to a quantitative bound, whose free boundary counterpart is the object of this paper. 
Here is our main result. 
\begin{theorem} \label{thm:main} 
    Let $(M^{n+1},g)$ be a compact Riemannian manifold of dimension $3\le n+1\le 7$ with (possibly empty) boundary. Then, given $\CA>0$, there exists a constant $C_\CA>0$ depending only on $(M,g)$ and $\CA$ such that the following property holds.
    
    For every smooth, compact, properly embedded, free boundary minimal hypersurface $\ms^n\subset M$ with $\Haus^n(\ms)\le \CA$, we have
    \begin{equation*} \label{eq:bound-lower-dimensions}
        \sum_{k=0}^n b_k(\ms)
		+\sum_{k=0}^{n-1} b_k(\partial \ms)
			\leq C_\CA(1+\ind(\ms)).
    \end{equation*}
    Moreover, for $n+1=3$, the constant $C_\CA$ can be chosen as $C_\CA = C \CA$ for some constant $C>0$ depending only on $(M,g)$.
\end{theorem}

\begin{remark}
We state the theorem with the assumption that $\ms$ is properly embedded, namely $\partial\ms=\ms\cap\partial\amb$, since there are some subtleties in defining the Morse index for nonproperly embedded free boundary minimal hypersurfaces (see \cite{CarlottoFranz2020}*{Section~2.2 and Appendix~A}). However, the same result holds for \emph{almost properly embedded} free boundary minimal hypersurfaces, arising via a min-max procedure in an arbitrary ambient manifold (see \cite{LiZhou2021}, in particular Definition~2.6 therein), with the notion of index used in the min-max theory, as e.g. in \cite{GuangLiWangZhou2021}.

Indeed, the only point in which we use effectively the definition of index is in the curvature estimates \cref{thm:curvature-estimates}, which holds also for almost properly embedded free boundary minimal hypersurface by \cite{GuangWangZhou2021}*{Theorem~1.4}.
\end{remark}

\begin{remark}
In the case $\partial \amb = \emptyset$, by properly embedded free boundary minimal hypersurfaces we mean embedded minimal hypersurfaces with $\partial \ms = \emptyset$. Under this convention, in the close case we just recover \cite{Song2019}*{Theorem~1}.
\end{remark}

\begin{remark} \label{rmk:InteriorBoundsBoundary}
In the case of orientable hypersurfaces we can use Poincaré--Lefschetz duality \cite{Hatcher2002}*{Theorem~3.43} and the exactness of the long exact sequence in cohomology associated to the pair $(\ms,\partial\ms)$ (cf. \cite{Hatcher2002}*{pp.~199--200}, see also Theorem~2.16 therein) to obtain that
\[
    \sum_{k=0}^{n-1} b_k(\partial\ms)
        \leq 2\sum_{k=0}^n b_k(\ms).
\]
We decided to state the theorem this way so that it also applies to nonorientable hypersurfaces and to highlight the fact that the Betti numbers of the boundary are also controlled by the index (and the area). However, we emphasize that $\sum_{k=0}^n b_k(\ms)$ contains already information on the topology of the boundary $\partial \ms$.
This is a consequence of \cref{cor:acyclic-cover} and \cref{prop:write-as-graph}\ref{wag:ii} in our proof.
\end{remark}

{
\begin{remark}
Note that, in the case when $M^{n+1}$ is a strictly two-convex subset of $\R^{n+1}$, Ambrozio--Carlotto--Sharp prove in \cite{AmbrozioCarlottoSharp2018IndexFBMS}*{Theorem~F} that
\[
b_1(\ms) + b_{n-1}(\ms) \le n(n+1)\ind(\ms).
\]
Indeed, we have that $b_1(\ms) = H_{n-1}(\ms,\partial\ms)$ and $b_{n-1}(\ms) = H_1(\ms,\partial\ms)$ (see \cite{AmbrozioCarlottoSharp2018IndexFBMS}*{Section~3}). 
Note that in our case we get an estimate on all the Betti numbers, on the other hand our constant $C_\CA=C_\CA(M,g,\CA)$ is not explicit and depends also on the area bound $\CA$.
\end{remark}
}

\subsection{Idea of the proof}
A recurrent method for obtaining quantitative estimates relating index and topology of a minimal hypersurface is {to rely} on the use of harmonic one-forms to obtain negative directions for the second derivative of the area.
More precisely, given a minimal hypersurface $\ms^n\subset M^{n+1}$, by basic Hodge theory, the dimension of the space of harmonic one-forms is equal to the dimension of the first \emph{de Rham cohomology} group of $\ms$, which is the first Betti number of $\ms$. Then, in order to relate the index to the dimension of the space of harmonic one-forms, one uses that harmonic one-forms give (in some suitable sense) negative directions for the second derivative of the area (see e.g. \cite{AmbrozioCarlottoSharp2018Index}*{Propositions 1 and 2}). We refer to \cites{ChodoshMaximo2016,ChodoshMaximo2018,AmbrozioCarlottoSharp2018Index} for examples of applications of this principle, whose idea goes back at least to Ros in \cite{Ros2006}*{Theorems 15--18}.
The drawback of this method is that, to turn the aforementioned relation between the first Betti number and the index of a minimal hypersurface into an effective bound, strong curvature assumptions on the ambient manifold $(M^{n+1},g)$ are imposed.

In order to deal with any closed ambient manifold $(M^{n+1},g)$ with $3\le n+1\le 7$, Song in \cite{Song2019} introduced a new approach, consisting of a geometric covering argument (see \cite{GrigoryanNetrusovYau2004} for other instances of covering arguments involving the Morse index).
The idea is to find a nice enough cover, whose size and overlap will bound the Betti numbers of a closed minimal hypersuface $\ms^n\subset\amb$.
    The way to choose this cover is to work at a scale of (appropriately quantified) stability (see \eqref{eq:defstabpre} below).
    Then one proves that it is possible to find sufficiently many unstable disjoint subsets of $\ms$, where sufficiently many means a definite fraction of the total number of elements in the cover. This is how the Morse index enters the picture.
    Song's approach to estimate Betti numbers could thus be seen as based on \emph{\v{C}ech cohomology}, as opposed to the previously described method which is based on the de Rham cohomology and Hodge theory.
    One of the main advantages is that this approach makes no assumption on the ambient geometry, but in turn it introduces a dependence on the area regardless of curvature assumptions.
     
    As anticipated, in this paper we take this same approach and we generalize it to the free boundary setting.
    Let us describe the structure of the proof of \cref{thm:main} in some more detail. 

First, we consider the cover $\mathcal{U}$ of $M$ consisting of \emph{stability balls} $B(p,s_\ms(p))$ for $p\in M$. Here, the \emph{stability radius} $s_\ms(p)$ depends on two fixed parameters $\bar r>0$ and $\lambda>0$, and it is defined as
\begin{equation} \label{eq:defstabpre}
s_\ms(p)\eqdef \sup\{r\le \bar r\st \text{$\ms$ is stable in $B(p,\lambda r)$}\},
\end{equation}
(see \cref{def:stability-radius}). The constants $\bar r,\lambda$ are chosen sufficiently small and sufficiently large respectively, in such a way that
\begin{enumerate} [label={\normalfont(\roman*)}]
    \item\label{pwc} $\ms\cap B(p,s_\ms(p))$ is well-controlled thanks to the curvatures estimates for stable (free boundary) minimal hypersurfaces;
    \item\label{pus} $\ms\cap B(p,\mu s_\ms(p))$ is unstable for all $\mu>\lambda$, if $s_\ms(p)<\bar r$.
\end{enumerate}

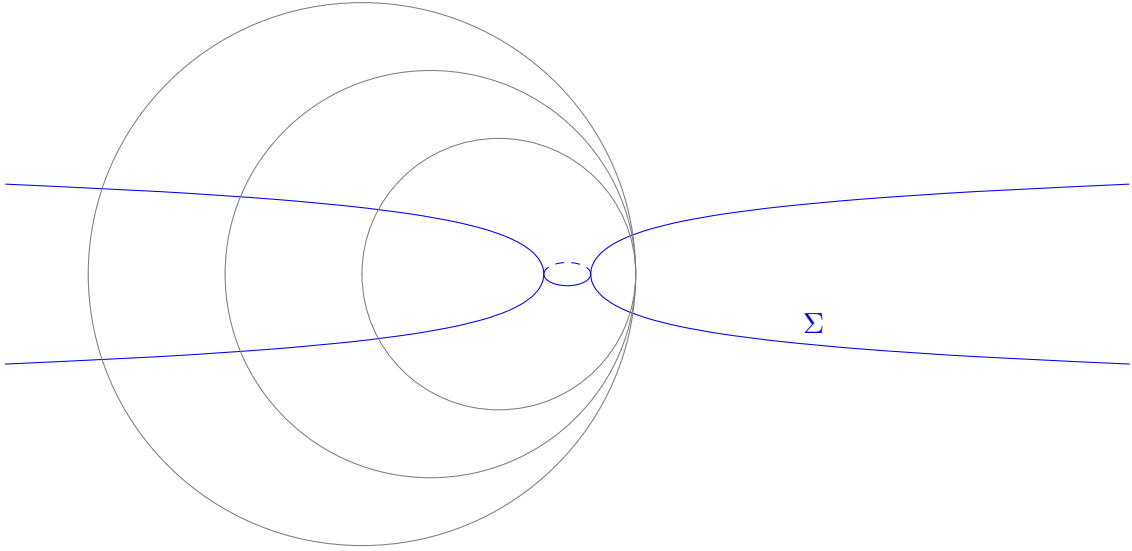
\begin{figure}[htpb]
\centering
\begin{tikzpicture}[line cap=round,line join=round,scale=0.9]

\pgfmathsetmacro{\R}{2.4}
\pgfmathsetmacro{\T}{1.19968}
\pgfmathsetmacro{\a}{1/(\T*cosh(\T))}

\pgfmathsetmacro{\l}{0.31}
\begin{scope}[myBlue]
\node at(1.5*\R,-0.3*\R) {$\ms$};
\draw[scale=1,domain=-\T:\T,smooth,variable=\s]  plot ({\l*\R*\a*cosh(\s/\l)},{\R*\a*\s});
\draw[scale=1,domain=-\T:\T,smooth,variable=\s]  plot ({-\l*\R*\a*cosh(\s/\l)},{\R*\a*\s});
\coordinate(P)at({\l*\R*\a},{0});
\draw[dashed] (P) arc (0:180:{\l*\R*\a} and {0.5*\l*\R*\a});
\draw (P) arc (0:-180:{\l*\R*\a} and {0.5*\l*\R*\a});
\end{scope}

\draw[gray] (1,0) arc (0:180:4);
\draw[gray] (1,0) arc (0:-180:4);

\draw[gray] (1,0) arc (0:180:3);
\draw[gray] (1,0) arc (0:-180:3);

\draw[gray] (1,0) arc (0:180:2);
\draw[gray] (1,0) arc (0:-180:2);
\end{tikzpicture}
\caption{Heuristic picture of the stability balls around a catenoidal neck.} \label{fig:catenoid}
\end{figure}

Now, the idea is to find a subset $\mathcal{B}\subset\mathcal{U}$ whose size is greater than or equal to (up to a constant) the sum of the Betti numbers of $\Sigma$ (which can be proved thanks to \ref{pwc}), and such that $\mu B$ and $\mu B'$ are disjoint for all $B,B'\in\mathcal{B}$, for some $\mu>\lambda$.
By observation \ref{pus}, this implies that
\[
\sum_{k=0}^nb_k(\Sigma)\lesssim\abs{\mathcal{B}} \le \ind(\ms),
\]
which would conclude the proof.

In order to find such subset $\mathcal{B}$, we first have to remove a region, called \emph{almost conical region} and denoted by $\conical$, which consists of a finite union of topological annuli where (in a suitable chart) the components of $\ms$ are close to a cone with bounded curvature and such that the stability radius of the center of the annulus is small.
The picture that one should have in mind is a catenoidal neck with large curvature. Heuristically, the index is concentrated at the catenoidal neck, hence the $\lambda$-dilations of the stability balls all contain the center of the neck (see \cref{fig:catenoid}). Thus, it is hard in this case to extract a subcover of stability balls whose $\lambda$-dilations are disjoint. The point is then to first remove an annulus around the catenoidal neck, where we know for other reasons that the topology is controlled.

The general idea is similar. Heuristically (see also \cref{fig:heuristic}), the index concentrates around a finite number of points, and away from those points the curvature and the topology are controlled (cf. \cite{Sharp2017}*{Theorem~2.3}). Thus, we would like to deal with the two regions separately.

\begin{figure}[htpb]
\centering

\begin{tikzpicture}[scale=1]
\coordinate (A1) at (0,0); 
\coordinate (B1) at (0.2,0.05); 
\coordinate (C1) at (0.4,0.12);
\coordinate (D1) at (0.6,0.21);
\coordinate (E1) at (0.8,0.35);

\coordinate (F1) at (4,1.97);
\coordinate (G1) at (4.1,2.05);
\coordinate (H1) at (4.3,2.17);
\coordinate (I1) at (4.5,2.24);
\coordinate (J1) at (4.7,2.28);
\coordinate (K1) at (4.85,2.27);

\coordinate (F2) at (5,-4.13);
\coordinate (G2) at (5.1,-4.13);
\coordinate (H2) at (5.3,-4.12);
\coordinate (I2) at (5.5,-4.11);
\coordinate (J2) at (5.7,-4.09);
\coordinate (K2) at (5.8,-4.07);

\coordinate (C2) at (0.7,-3);
\coordinate (D2) at (0.9,-3.05);
\coordinate (E2) at (1.1,-3.1);

\coordinate (X) at (-1.5,-0.2);
\coordinate (Y) at (-1.2, -3.2);
\coordinate (U) at (7.5, 1.8);
\coordinate (V) at (7.9, -2.6);

\fill[gray!10!white] plot [smooth cycle, tension=0.6] coordinates {(A1) (B1) (C1) (D1) (E1)
			      (1.5,0.9) (2,1.2) (2.5,1.3) (3,1.7) (3.7,1.71)
			      (F1) (G1) (H1) (I1) (J1) (K1)
			      (5.5,2.2) (6,1.8)
			      (U) (8.5,-0.5) (V)
			      (7,-2.9) (6.5, -3.3) (6.2, -3.93) 
			      (K2) (J2) (I2) (H2) (G2) (F2)
			      (4.5,-4) (3.5,-3.2) (2.7, -2.5) (2,-2.7) (1.5, -3.2) 
			      (E2) (D2) (C2)
			      (0.3, -3)
			      (Y) (-1.9,-2) (X) (-0.5,-0.05)};

\pgfmathsetmacro\R{1}
\fill[white] (-0.4, -3.15) arc (180:-180:\R);
\fill[white] (5.685, -1.2) circle (\R);

\draw plot [smooth cycle, tension=0.6] coordinates {(A1) (B1) (C1) (D1) (E1)
			      (1.5,0.9) (2,1.2) (2.5,1.3) (3,1.7) (3.7,1.71)
			      (F1) (G1) (H1) (I1) (J1) (K1)
			      (5.5,2.2) (6,1.8)
			      (U) (8.5,-0.5) (V)
			      (7,-2.9) (6.5, -3.3) (6.2, -3.93) 
			      (K2) (J2) (I2) (H2) (G2) (F2)
			      (4.5,-4) (3.5,-3.2) (2.7, -2.5) (2,-2.7) (1.5, -3.2) 
			      (E2) (D2) (C2)
			      (0.3, -3)
			      (Y) (-1.9,-2) (X) (-0.5,-0.05)};

\begin{scope}[blue]
\draw (F1) to [bend left=20] (F2);
\draw (G1) to [bend left=20] (G2);

\draw plot [smooth, tension = 0.5] coordinates {(H1) (5.2,0.2) (5.45,-1) (5.57, -1.2) (5.65, -1) (5.4,0.2) (I1)};
\draw plot [smooth, tension = 0.5] coordinates {(H2) (5.5,-2.5) (5.5,-1.5) (5.6, -1.3) (5.7, -1.5) (5.7,-2.5) (I2)};
\pgfmathsetmacro{\t}{atan(0.03/0.1)}
\begin{scope}[rotate around={\t:(5.57,-1.2)}]
\draw[gray] (5.57, -1.2) arc (90:270:0.03 and {sqrt(0.03*0.03+0.1*0.1)/2});
\draw (5.57, -1.2) arc (90:-90:0.03 and {sqrt(0.03*0.03+0.1*0.1)/2});
\end{scope}

\draw plot [smooth, tension = 0.5] coordinates {(J1) (5.25,1.2) (5.68,0) (5.9,-1.29) (5.97, -1.45) (6, -1.27) (5.8,0) (5.37,1.2) (K1)};
\draw plot [smooth, tension = 0.5] coordinates {(J2)  (5.9,-2.7) (5.93, -1.7) (5.97, -1.53) (6.03,-1.7) (6.02, -2.7) (K2)};
\draw[gray] (5.97, -1.45) arc (90:270:0.02 and {0.04});
\draw (5.97, -1.45) arc (90:-90:0.02 and {0.04});

\draw[gray] (0.45, -2.9) arc (90:182:0.02 and {0.08});
\draw (0.45, -2.9) arc (90:0:0.02 and {0.08});
\draw plot [smooth, tension = 0.55] coordinates {(A1) (0.35,-1.2) (0.35,-2.65) (0.45, -2.9) (0.55, -2.65) (0.55,-1.2) (B1)};

\pgfmathsetmacro{\t}{atan(0.01/0.1)}
\begin{scope}[rotate around={-\t:(0.87, -2.6)}]
\draw[gray] (0.87,-2.6) arc (90:270:0.03 and {sqrt(0.01*0.01+0.1*0.1)/2});
\draw (0.87,-2.6) arc (90:-90:0.03 and {sqrt(0.01*0.01+0.1*0.1)/2});
\end{scope}
\draw plot [smooth, tension = 0.55] coordinates {(C1) (0.75,-1.2) (0.77,-2.4) (0.87, -2.6) (0.97, -2.4) (0.95,-1.2) (D1)};
\draw plot [smooth, tension = 0.9] coordinates {(C2) (0.75, -2.8) (0.86, -2.7) (0.93, -2.8) (D2)};
\draw (E1) to [bend left = 15] (E2);

\node at (1.5,-0.6) {$\ms$};
\end{scope}

\pgfmathsetmacro\r{0.5}
\draw[gray] (-0.4, -3.15) arc (180:2:\R);
\draw[gray] (5.685, -1.2) circle (\R);

\draw[gray,->] (1.1,-3.2) to [bend right=20] (7,-5);
\draw[gray,->] (6.5,-1.92) to [bend left=20] (9.5,-4.42);
\node[text width=4cm,below] at (9.5,-4.6) {Region where the index ``concentrates''};

\draw[gray,->] (7,1) to [bend left=20] (9.5,1);
\node[text width=3cm,right] at (9.8,0.5) {Region where the curvature (and so the topology) is controlled};
\end{tikzpicture}
\caption{Heuristic picture of the almost conical region $\conical$ (in gray) and the concentration region $M\setminus\conical$.} \label{fig:heuristic}
\end{figure}
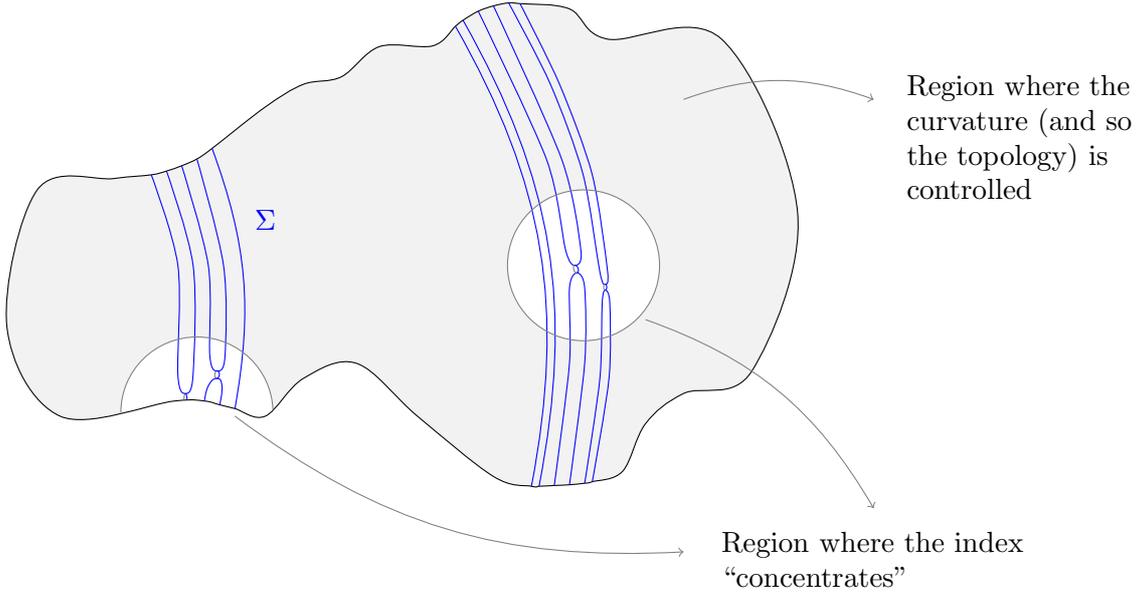

We define (in \cref{sec:conical-region}) an almost conical region $\conical$ such that 
\[
\sum_{k=0}^nb_k(\Sigma) \sim \sum_{k=0}^nb_k(\ms\setminus\conical).
\]
The advantage is that (a suitable subset of) the stability balls centered in $M\setminus\conical$ (whose cardinality controls the sum of the Betti numbers of $M\setminus\conical$) have then, heuristically, a ``binary tree structure''. One should think that the vertices of the tree are the stability balls and an edge between two balls means that their $\lambda$-dilations intersect. The leaves of this binary tree are stability balls whose $\lambda$-dilations are disjoint, and their number is comparable with the number of vertices of the tree (this is a standard property of binary trees). The details are contained in \cref{sec:concentration-region}.

The one described above is a very general outline and it is essentially the same as in the closed case (see \cite{Song2019}*{Section~1} for more details).
However, there are unfortunately many technical details to take care of. 
The main further subtleties in dealing with the free boundary setting regard the construction of the almost conical region. Indeed, in our case with boundary, the almost conical region consists of annuli contained in the interior of $M$ (which behave as in the closed case) but also ``half-annuli'' centered at the boundary $\partial M$. We remark that we are able to avoid annuli that intersect the boundary but are not centered at boundary points.

\begin{case3}
As stated in \cref{thm:main}, in the case $n+1=3$ we are able to keep track of the dependence on the area of the constant in the theorem. The proof differs basically only in one step, namely \cref{thm:BoundaryNonConcentration}. Hence, we decided to described the adaptations to the case $n+1=3$ alongside the higher dimensional case, by using boxes like this one throughout the paper.
\end{case3}

\subsection{Higher dimensional results}
In \cite{Song2019}, Song also proves an estimate on the size of the singular set of
codimension one stationary integral varifolds given that this set is already a priori small (see \cite{Song2019}*{Theorem~2}).
More precisely, if $\ms^n\subset \amb^{n+1}$ is a compactly supported, properly embedded, multiplicity one
stationary integral varifold such that $\Sing(\ms)$ has Hausdorff dimension at most $n-7$ and mass bounded by $\CA$, then
    \begin{equation} \label{eq:bound-higher-dimensions}
        \Haus^{n-7}(\Sing(\ms))
			\leq C_\CA(1+\ind(\ms))^{7/n},
    \end{equation}
where $C_\CA=C_\CA(M,g,\CA) >0$ is a constant depending only on $(M,g)$ and $\CA$.

One of the main tools in lower dimensions was the use of curvature estimates, which are not available in the higher dimensional setting. Instead, Song used a version of \cite{NaberValtorta2020}*{Theorem~1.6} but assuming stability (and a priori smallness of the singular set and bounded area) instead of area-minimizing.
The main consequence is that
\begin{equation}
\label{eq:NaberValtortaEdelen-estimate}
    \Haus^{n-7}(\Sing(\ms)\cap B(p,s_\ms(p)) \leq C_{NV} = C_{NV}(M,g,\CA),
\end{equation}
where the definition of $s_\ms$ is the same as in \eqref{eq:defstabpre}, assuming $\lambda=2$.

A generalization of  \cite{NaberValtorta2020}*{Theorem~1.6} for free boundary area-minimizing codimension one integral currents was proven by Edelen in \cite{Edelen2020}*{Theorem~2.2} (see also Theorem~0.5 therein).
We expect that one can substitute the hypothesis of free boundary area-minimizing by assuming an a priori smallness of the size of the singular set, an area bound and stability.
This would allow us to get an estimate like~\eqref{eq:NaberValtortaEdelen-estimate} in the free boundary case.
Having such a result, we think that it should be possible to prove a result analogous to \eqref{eq:bound-higher-dimensions} but in the free boundary case, following the same lines as in \cite{Song2019}*{Section~5.2}.
This problem might be addressed elsewhere.

\subsection*{Acknowledgements}
This project originated from the Master thesis of the first author under the supervision of Alessandro Carlotto, whom we would like to thank.
We would also like to thank Antoine Song for many helpful discussions and in particular for pointing out a mistake in the first version of this paper. Moreover, we thank Nick Edelen for kindly answering our questions, {and the anonymous referee for helpful comments and suggestions.}
The first author was supported by a US--Spain Fulbright grant.
This project has received funding from the European Research Council (ERC) under the European Union’s Horizon 2020 research and innovation programme (grant agreement No. 947923), and was also partially supported by the NSF award DMS-2104229 and by Simons Foundation International, LTD.

\subsection{Structure of the article}

We provide an outline of the contents of this article by means of the following diagram.

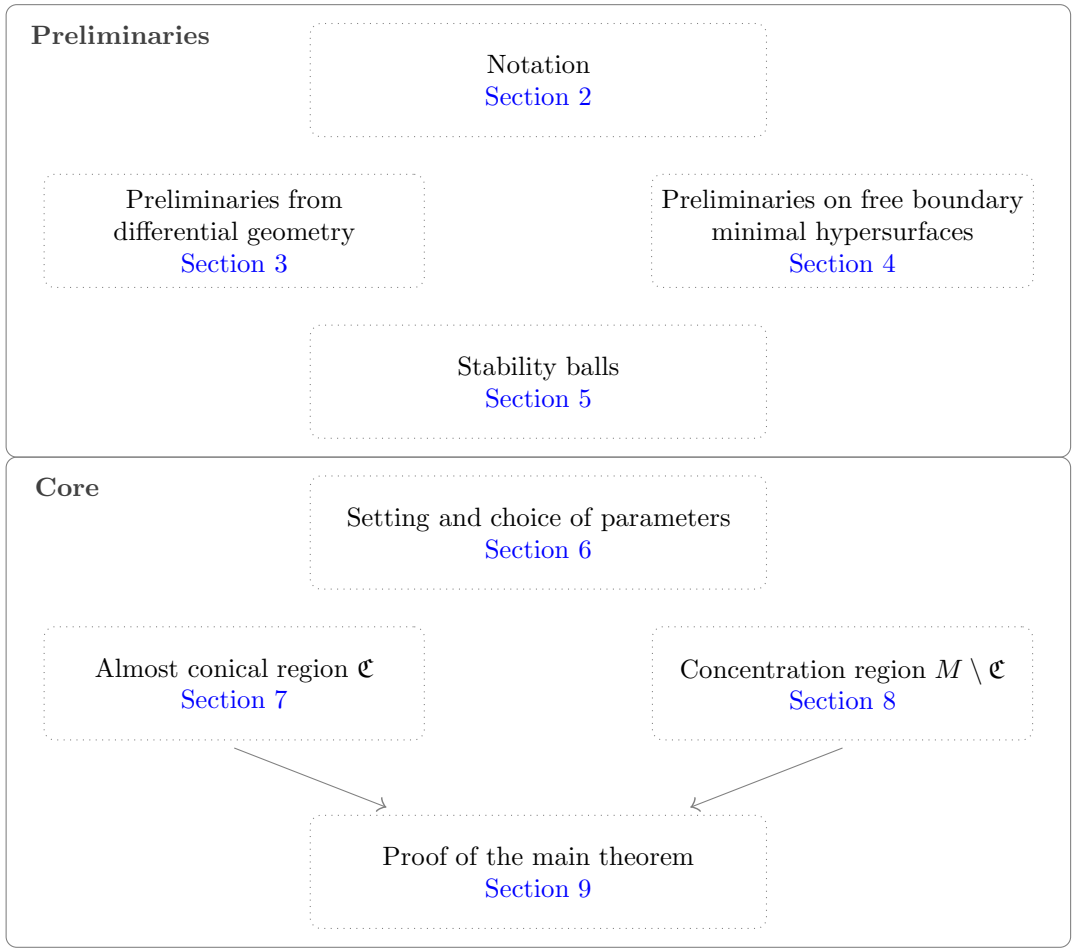
\begin{figure}[htpb]
\usetikzlibrary{trees}
\tikzstyle{every node}=[anchor=west]
\tikzstyle{sec}=[shape=rectangle, rounded corners,draw=gray, dotted,minimum width = 6cm,minimum height = 1.5cm,yshift=-1cm]
\tikzstyle{lat}=[xshift=7cm]
\tikzstyle{optional}=[dashed,fill=gray!50]
\tikzset{bigarr/.style={
decoration={markings,mark=at position 1 with {\arrow[scale=1.7]{>}}},
postaction={decorate}}}
\centering
\begin{tikzpicture}[align=center, font=\small]
\node[sec] (not) {Notation \\ \cref{sec:notation}};

\node [gray!50!black,left of =not, xshift=-4.5cm, yshift=.6cm] {\small \textbf{Preliminaries}};

\draw[gray,rounded corners] ($(not)+(-7,1)$)  rectangle ($(not)+(7,-5)$);

\node[sec] [below of=not,xshift=-4cm,minimum width = 5cm] (prediff) {Preliminaries from \\ differential geometry\\ \cref{sec:pre-diff-geom}}; 
\node[sec] [below of=not,xshift=4cm,minimum width = 5cm] (prefbms) {Preliminaries on free boundary \\ minimal hypersurfaces\\ \cref{sec:pre-fbms}}; 

\node[sec] [below of=prediff, xshift=4cm] (stab) {Stability balls \\ \cref{sec:stability}}; 

\node[sec] [below of=stab] (sett) {Setting and choice of parameters\\ \cref{sec:choice-param}};

\node [gray!50!black,left of =sett, xshift=-5.2cm, yshift=.6cm] {\small \textbf{Core}};

\node[sec] [below of=sett,xshift=-4cm,minimum width = 5cm] (conical) {Almost conical region $\conical$ \\ \cref{sec:conical-region}};

\node[sec] [below of=sett,xshift=4cm,minimum width = 5cm] (concentration) {Concentration region $M\setminus\conical$ \\ \cref{sec:concentration-region}}; 

\node[sec] [below of=conical,xshift=4cm,yshift=-.5cm] (proof) {Proof of the main theorem \\ \cref{sec:proof}}; 
\draw[bigarr,gray] ([yshift=-1mm]conical.south) --([xshift=-20mm,yshift=1mm]proof.north);
\draw[bigarr,gray] ([yshift=-1mm]concentration.south) --([xshift=20mm,yshift=1mm]proof.north);

\draw[gray,rounded corners] ($(proof)+(-7,5.5)$)  rectangle ($(proof)+(7,-1)$);

\end{tikzpicture}
\end{figure}

\section{Notation}\label{sec:notation}

Let us denote by $\R^{n+1}_+ \eqdef \{(x^1,\ldots, x^{n+1})\in\R^{n+1}\st x^{n+1}\ge 0\}$ the half-space in $\R^n$, and by $\partial \R^{n+1}_+ = \{(x^1,\ldots, x^{n+1})\in\R^{n+1}\st x^{n+1} = 0\}$ its boundary. Moreover, let $\omega_n$ be the volume of the $n$-th dimensional Euclidean unit ball.

We write $b_k(\ms)$ for the $k$-th Betti number of 
a differentiable manifold $\ms$, which we define as the dimension of the $k$-th de Rham cohomology group. Betti numbers are usually defined as the rank of the singular homology groups with $\Z$ coefficients, but this turns out to be the same, by the Universal Coefficients theorems for homology and cohomology (cf. \cite{Hatcher2002}*{Theorems~3A.3 and~3.2}) and de Rham's theorem (see \cite{Warner1983}*{Theorem~4.17}). 

When we write $C = C(M,g)$, we mean that the constant $C$ depends only on $(M,g)$. We use a similar convention to express dependency on other constants.

We put two capital letters as subscript for better reference to the several constants that appear in the paper. We decided to do so, because this allows us to track the dependence of $C_\CA$ in \cref{thm:main} on the other constants (see \eqref{eq:dep-const}).
The following table shows what the different names of the constants stand for.

\vspace{2ex}
\begin{center}
\begin{tikzpicture}
\node (table) [inner sep=0pt] {
\bgroup
\def\arraystretch{1.3}
\begin{tabular}{  p{5ex} p{41ex} | p{5ex} p{41ex}}
$C_{AC}$ & Acyclic Covering, \cref{lem:cover-and-betti}  &    $C_{MF}$ & Monotonicity Formula, \cref{thm:monotonicity}\\
  $C_{BC}$ & Besicovitch Covering, \cref{thm:Besicovitch} & $C_{DE}$ & Density Estimate,  \cref{cor:density-estimate}  \\
  $C_{VB}$ & Volume Bound,  \cref{lem:Besicovitch-type-lemma}  &  $C_{CE}$ & Curvature Estimates, \cref{thm:curvature-estimates}\\
  $C_{LF}$ & Layered Family, \cref{prop:we-can-layer} &  $C_{CS}$ &  Components Stability, \cref{prop:sheeting-results}  \\
  $C_{PL}$ & Previous Layers, \cref{lem:fact1} &  $C_{CT}$ & Components Telescopes,  \cref{prop:ConicalRegion} \\
 $C_{NL}$ & Next Layers, \cref{lem:fact2} & $C_{MI}$ & Morse Index, \cref{thm:DisjointIndex}  \\
\end{tabular}
\egroup
};
\draw [rounded corners] (table.north west) rectangle (table.south east);
\end{tikzpicture}
\end{center}
\vspace{1ex}


\section{Preliminaries from differential geometry} \label{sec:pre-diff-geom}

In this section, we collect some preliminary definitions and results from differential geometry.

\subsection{Convexity radius}
Let us give the definition of convexity radius that we use in this paper in place of the injectivity radius, which is not universally defined in manifolds with boundary.
Let us assume that $(M^{n+1},g)$ is a compact Riemannian manifold with (possibly empty) boundary.

\begin{definition} \label{def:convex-set}
We say that $U\subset M$ is a \emph{convex subset of $M$} if for all $p, q\in U$ there is a \emph{unique} length-minimizing curve connecting $p$ and $q$ in $M$ and this length-minimizing curve is contained in $U$.
\end{definition}
{
\begin{remark}
Here, by length-minimizing curve connecting $p,q\in M$, we mean a curve realizing the distance $d(p,q)$ (defined as the infimum of the lengths of curves connecting $p$ and $q$). Note that such a curve is not necessarily a geodesic in the standard Riemannian geometric sense, {because it could have a contact set with $\partial M$ (behaving like an obstacle, see \cite{AlexanderBergBishop1987}).}
Moreover, given $p\in M$ and $r>0$, $B(p,r)\eqdef \{q\in M \st d(p,q)<r\}$ is not necessarily the image of the ball of radius $r$ via the exponential map centered in $p$.
We refer to \cite{AlexanderBergBishop1987} and \cite{AlexanderBergBishop1993} for a discussion about local distance geometry in manifolds with boundary. 
\end{remark}
}
\begin{remark} \label{rem:convex-trivial}
Note that every convex subset is simply connected, {because of the uniqueness assumption in \cref{def:convex-set}}.
\end{remark}

\begin{definition} \label{def:convexity-radius}
Given any point $x\in M$, we define the \emph{convexity radius $\conv_M(x)$ of $M$ at $x$} as
\[
\conv_M(x) \eqdef \sup\{r>0 \st B(x,r) \text{ is a convex subset of $M$}\}.
\]
Moreover, we define the \emph{convexity radius $\conv_M$ of $M$} as $\conv_M\eqdef \inf_{x\in M} \conv_M(x)$.
\end{definition}

\begin{proposition}
The convexity radius on $(M^{n+1},g)$ is bounded away from $0$, namely $\conv_M>0$.
\end{proposition}
\begin{proof}
The proof follows from \cite{AlexanderBergBishop1987}*{Theorem 5} and \cite{AlexanderBergBishop1993}*{Corollary 2}, since $M$ is compact.
\end{proof}

\subsection{Almost acyclic covers and Betti numbers}
In this subsection we define the notion of almost acyclic cover and we recall a topological lemma relating an almost acyclic cover to the Betti numbers of a manifold.

\begin{definition} \label{def:acyclic-overlap}
    Let $\ms^n$ be a smooth differentiable manifold. We say that an open cover $\mathcal{U}$ of $\ms$ is \emph{$\alpha$-almost acyclic} if every {nonempty} finite intersection of elements of $\mathcal{U}$ has sum of the Betti numbers less or equal than $\alpha$. 
    If $\alpha = 1$, we say that the cover is \emph{acyclic} (e.g. the only possibly nontrivial Betti number is $b_0(\ms)$).

    Moreover, we say that the \emph{overlap} of a cover $\mathcal{U}$ is at most $\beta\in\N$ if any element of $\mathcal{U}$ intersects at most $\beta$ other distinct elements of $\mathcal{U}$.
\end{definition}

\begin{lemma} [{cf. \cite{Song2019}*{Lemma 26}}] \label{lem:cover-and-betti}
    Given $n,\alpha,\beta\in\N$, there exists $C_{AC}=C_{AC}(n,\alpha,\beta)>0$ such that the following statement holds. Let $\ms^n$ be a smooth differentiable manifold that admits an open $\alpha$-almost acyclic cover $\mathcal{U}$ with overlap at most $\beta$. Then it holds
    \begin{equation*} \label{eq:topological-bound}
        \sum_{k=0}^n b_k(\ms)
        \leq C_{AC}\abs{\mathcal{U}}.
    \end{equation*}
\end{lemma}
\begin{proof}
    The proof is essentially an application of Mayer--Vietoris theorem for cohomology, and fits in the framework of \v{C}ech cohomology.
    It is carried out in detail in the Appendix of \cite{Song2019} (see also \cite{BallmannGromovSchroeder1985}*{Lemma~12.12} for the result in the case of acyclic covers).
\end{proof}

\subsection{Besicovitch theorem}

Here we recall Besicovitch theorem, which is useful to extract subcovers with controlled overlap.

\begin{theorem} [Besicovitch theorem]\label{thm:Besicovitch}
    Let $(\amb^{n+1},g)$ be a compact Riemannian manifold with (possibly empty) boundary. Then there exist an integer $C_{BC} = C_{BC}(\amb,g)\ge 1$ and a constant $0<r_{BC}= r_{BC}(M,g) <\conv_M $ such that the following property holds.
	
	Let $\mathcal{F}$ be a collection of (nondegenerate) balls with radius at most $r_{BC}$ and let $S$ be the set of centers of the balls in $\mathcal{F}$. Then there exist $C_{BC}$ subfamilies $\mathcal{B}^{(1)},\ldots,\mathcal{B}^{(C_{BC})}\subset\mathcal{F}$ of disjoint balls in $\mathcal{F}$ such that
	\[
		S\subset \bigcup_{i=1}^{C_{BC}} \bigcup_{B\in\mathcal{B}^{(i)}} B.
	\]

\end{theorem}
\begin{proof}
    Manifolds with boundary have the property of local uniqueness of length-minimizing curves (see \cite{AlexanderBergBishop1987}*{Theorem 5}) and their metric spheres are compact.
    These two facts imply the property of being directionally limited as defined in \cite{Federer1969}*{Section~2.8.9}. Hence, the theorem follows from \cite{Federer1969}*{Theorem~2.8.14}.
\end{proof}

\subsection{Layered family of balls}
In this section, let us assume that $(M^{n+1},g)$ is a compact Riemannian manifold with (possibly empty) boundary.
Later, it will be convenient to divide families of the geodesic balls of $M$ into subfamilies of balls with comparable radii. To this purpose, we introduce some definitions and notation, and we recall some properties.

\begin{definition}
Let $\mathcal{F}$ be the set of geodesic balls in $M$ of radius less than $1$. 
For each integer $k\ge 0$, define the \emph{$k$-layer} of $\mathcal{F}$ as
\begin{equation*}
    \mathcal{F}_k \eqdef \{ B(x,r)\in\mathcal{F} \st r\in[2^{-(k+1)}, 2^{-k})\}.
\end{equation*}
For any subset $\mathcal{B}\subset\mathcal{F}$ of geodesic balls, we denote by $\mathcal{B}_k=\mathcal{B}\cap \mathcal{F}_k$ the $k$-layer of $\mathcal{B}$.
\end{definition}

\begin{definition} \label{def:layered}
Let $\mathcal{B}\subset\mathcal{F}$ be a family of geodesic balls with radius less than $1$ and let $\mu\ge 1$.
\begin{enumerate} [label={\normalfont(\roman*)}]
\item We say that $\mathcal{B}$ is \emph{$\mu$-disjoint} if, for all $b,b'\in\mathcal{B}$, either $b=b'$ or $\mu b\cap \mu b'=\emptyset$.

\item We say that $\mathcal{B}$ is \emph{layered $\mu$-disjoint} if $\mathcal{B}\cap\mathcal{F}_k$ is $\mu$-disjoint for all $k\ge 0$.
\end{enumerate}
\end{definition}

\begin{lemma} \label{lem:Besicovitch-type-lemma}
Given $\lambda\ge 2$,  there exist $0<\bar{r} = \bar{r}(M,g,\lambda)<1$ small enough and a constant $C_{VB}=C_{VB}(M,g,\lambda)\ge1$ such that the following properties hold.
    \begin{enumerate} [label={\normalfont(\roman*)}]
    \item\label{btl:i} Let $B$ be a ball of radius at most $\bar{r}$ and let $\mathcal{B}$ be a family of disjoint balls with radius at most~$\bar{r}$ and satisfying that $\radius(B)/2\leq \radius(b) \leq  2\radius(B)$ for all $b\in \mathcal{B}$.
    Then we have
	\[
	    \abs{\{ b\in \mathcal{B} \st 3\lambda b \cap 3\lambda B \neq\emptyset \}} \leq C_{VB}.
	\]
	\item\label{btl:ii} Let $\mathcal{B}$ be a family of disjoint balls with radius \emph{equal to} $\bar r$, then $\abs{\mathcal{B}}\leq C_{VB}$.
	\end{enumerate}
\end{lemma}
\begin{proof}
    Let $\mathcal{F}=  \{ b\in\mathcal{B} \st 3\lambda b \cap 3\lambda B \neq\emptyset \}$ be the subcollection of balls whose size we want to estimate in \ref{btl:i}.
    Let $x_b$ be the center of $b\in \mathcal{F}$ and $x_B$ be the center of $B$.
    Then we see that 
    \[
        d(x_b,x_B)
            \leq 3\lambda \radius(b)+ 3\lambda \radius(B)
            \leq 9\lambda\radius(B).
    \]
    This means that $b$ is contained in the ball centered at $x_B$ and with radius $9\lambda\radius(B)+\radius(b)\le 10\lambda \radius(B)$, and hence certainly in $B' = 10\lambda B$.
    
    We also know that $b\supset B(x_b,\radius(B)/2)$. By the monotonicity of $\Haus^{n+1}$ with respect to inclusion, we have that
    \[
        \sum_{b\in\mathcal{F}} \Haus^{n+1}(B(x_b,\radius(B)/2))
            \leq \sum_{b\in\mathcal{F}} \Haus^{n+1}(b)
            = \Haus^{n+1}\left(\bigcup_{b\in\mathcal{F}} b\right)
            \leq \Haus^{n+1}(B'),
    \]
    where we used that the balls in $\mathcal{B}\supseteq \mathcal{F}$ are disjoint.
    
    If we consider the density $\Theta^M_g(p,r) = {\Haus^{n+1}(B(p,r))}/{(\omega_{n+1}r^{n+1})}$, we know that, if $p\in\amb\setminus\partial\amb$, then $\displaystyle \lim_{r\to 0^+} \Theta^M_g(p,r) = 1$ and that $\displaystyle \lim_{r\to 0^+} \Theta^M_g(p,r) = 1/2$ if $p\in\partial \amb$.
   	Therefore, by compactness of $M$, we can make $\bar{r} = \bar{r}(\amb ,g,\lambda)>0$ small enough so that
   	$1/4\leq \Theta^M_g(p,r) \leq 2$ for every $p\in\amb$ and $0<r\leq 20\lambda\bar{r}$ which means that
   	\[
   	    \frac{\omega_{n+1}r^{n+1}}{4}
   	        \leq \Haus^{n+1}(B(p,r))
   	        \leq 2\omega_{n+1}r^{n+1}.
   	\]
   	
   	From that and the previous estimates, we get
   	\[
        \abs{\mathcal{F}} \frac{\omega_{n+1}\radius(B)^{n+1}}{2^{n+3}}
            \leq 2\omega_{n+1}(10\lambda\radius(B))^{n+1},
    \]
    which implies
    \[
        \abs{\mathcal{F}}
            \leq 2^{n+4}(10\lambda)^{n+1}.
    \]
    
    In case \ref{btl:ii}, where every ball in $\mathcal{B}$ has radius $\bar{r}$, {as chosen above,}  we have that
    \[
        \abs{\mathcal{B}} \frac{\omega_{n+1}\bar{r}^{n+1}}{4}
            \leq \sum_{b\in\mathcal{B}}\Haus^{n+1}(b)
            = \Haus^{n+1}\left(\bigcup_{b\in\mathcal{B}} b\right)
            \leq \Haus^{n+1}(\amb).
    \]
    Therefore, it suffices to take $\displaystyle C_{VB} \eqdef \max\left\{ 2^{n+4}(10\lambda)^{n+1} ,\, \frac{4\Haus^{n+1}(\amb)}{\omega_{n+1}\bar{r}^{n+1}} \right\}$.
\end{proof}

Using this lemma we show that, given a finite disjoint family, we can extract a layered family in the sense of \cref{def:layered}, with size bounded below by a definite fraction of the size of the initial family.
\begin{proposition} \label{prop:we-can-layer}
Given $\lambda\ge 2$,  there exist $0<\bar{r} = \bar{r}(M,g,\lambda)<1$ small enough and a constant $C_{LF}=C_{LF}(M,g,\lambda)> 0$ such that, for any given finite family $\mathcal{B}$ of disjoint of balls with radius at most $\bar{r}$, we can find a layered $3\lambda$-disjoint family $\mathcal{B}'$ satisfying
    \[
        \abs{\mathcal{B}}\leq C_{LF} \abs{\mathcal{B}'}.
    \]
\end{proposition}
\begin{proof}
    Let $\bar r=\bar r(M,g,\lambda)>0$ small enough such that we can apply \cref{lem:Besicovitch-type-lemma}.
    We perform a greedy algorithm on each of the families $\mathcal{B}\cap \mathcal{F}_k$ to obtain a $3\lambda$-disjoint subfamily $\mathcal{B}_k'$ as follows:
    at each step we select a ball $B\in \mathcal{B}\cap \mathcal{F}_k$ to be in the output family $\mathcal{B}_k'$ and we discard all those balls $b\in \mathcal{B}\cap \mathcal{F}_k$ which satisfy $3\lambda b\cap 3\lambda B\neq \emptyset$.
    We stop when every ball has either been marked as discarded or selected, which eventually happens since $\mathcal{B}$ is finite.
    By \cref{lem:Besicovitch-type-lemma}\ref{btl:i} (note that balls in $\mathcal{F}_k$ satisfy the condition on having comparable radii by construction), we know that we are discarding at most $C_{VB}$ of them at each step.
    Thus, we can say that the maximum number of discarded balls is $C_{VB} \abs{\mathcal{B}_k'}$, where $\mathcal{B}_k'$ consists of the selected balls. Therefore
    \[
        \abs{\mathcal{B}\cap \mathcal{F}_k}
            \leq \abs{\mathcal{B}_k'}(C_{VB}+1).
    \]
    So it suffices to take $C_{LF} \eqdef C_{VB}+1$ and set $\mathcal{B}' \eqdef \bigcup_{k\in\N} \mathcal{B}_k'$.
\end{proof}

Finally, we recall a couple of lemmas about layered $3\lambda$-disjoint families of balls, which can be found in \cite{Song2019}*{Facts 1 and 2}, and for which we omit the proof.

\begin{lemma} \label{lem:fact1}
There exists $C_{PL}=C_{PL}(M,g)>0$ such that the following property holds.
Let $\mathcal{B}$ be a layered $3\lambda$-disjoint family of balls and consider $b\in \mathcal{F}_k$ with $k\ge 0$. Then, for all $0\le k'\le k$, the size of 
\[
\{ \hat b\in \mathcal{B}_{k'}\st 6\lambda \hat b\cap 3\lambda b \not=\emptyset \}
\]
is bounded above by $C_{PL}$.
\end{lemma}

\begin{lemma} \label{lem:fact2}
Given $\bar u\ge 0$, there exists $C_{NL} = C_{NL}(\bar u)= C_{NL}(M,g,\bar u) > 0$ such that the following property holds. Let $\mathcal{B}$ be a layered $3\lambda$-disjoint family of balls and consider $b\in \mathcal{F}_k$ with $k\ge 0$. Then the size of
\[
\{ \hat b\in \mathcal{B}_{k+u}\st 0\le u\le \bar u, \ 3\lambda \hat b\cap 3\lambda b \not=\emptyset \}
\]
is bounded above by $C_{NL}$.
\end{lemma}

\section{Preliminaries on free boundary minimal hypersurfaces} \label{sec:pre-fbms}

In this section, let us assume that $(M^{n+1},g)$ is a compact Riemannian manifold of dimension $n+1\ge 3$ with (possibly empty) boundary. Moreover, let us assume that $\ms^n\subset \amb^{n+1}$ is a smooth compact hypersurface that is \emph{properly embedded} in $M$, namely $\partial\ms=\ms\cap \partial \amb$.

\subsection{First and second variation formulas}

We denote by $\vf(M)$ the linear space of smooth ambient vector fields $X$ such that $X(p)\in T_p\partial M$ for all $p\in\partial M$.
The \emph{first variation} of the area of $\ms$ along a vector field $X\in\vf(M)$ is given by
\[
\frac{\d}{\d t}\Big|_{t=0} \Haus^n(\psi_t(\ms)) = -\int_\ms g(X,H)\de \Haus^n + \int_{\partial\ms} g(X,\eta) \de \Haus^{n-1},
\]
where $\psi\colon[0,\infty)\times M\to M$ is the flow generated by $X$, $H$ is the mean curvature of $\ms$ and $\eta$ is the unit conormal along $\partial\ms$.

\begin{definition}
    We say $\ms^n\subset M^{n+1}$ is a \emph{free boundary minimal hypersurface} if it is a critical point of the area functional with respect to variations induced by $\vf(\amb)$. Thanks to the formula above, this is equivalent (in the properly embedded setting) for $\ms$ to {have vanishing mean curvature and to meet} the ambient boundary orthogonally along its own boundary.
\end{definition}

From now on, let us assume that $\ms^n\subset M^{n+1}$ is a free boundary minimal hypersurface.
The \emph{second variation} of the area of $\ms$ along a vector field $X\in\vf(M)$ is given by 
\begin{align*}
\frac{\d^2}{\d t^2} \Big|_{t=0} \Haus^n(\psi_t(\ms))= Q^\ms(X^\perp,X^\perp) &\eqdef \int_\ms (\abs{\cov^{\perp}X^{\perp}}^2 - (\Ric_{\amb}(X^\perp,X^\perp) + \abs {A^\ms}^2 \abs{X^\perp}^2)) \de \Haus^n  +{} \\
&\phantom{=} {}+\int_{\partial\ms} \II^{\partial\amb}(X^\perp, X^\perp) \de \Haus^{n-1},
\end{align*}
where $X^\perp$ is the normal component of $X$, $\cov^\perp$ is the induced connection on the normal bundle $N\ms$ of $\ms\subset\amb$, $A^\ms$ is the second fundamental for of $\ms\subset M$, $\Ric_\amb$ is the ambient Ricci curvature and $\II^{\partial\amb}$ denotes the second fundamental form of the ambient boundary.
Here, again, $\psi\colon[0,\infty)\times M\to M$ is the flow generated by $X$.

\begin{definition}
    We define the \emph{Morse index} $\ind(\ms)$ of $\ms$ as the maximal dimension of the sections $\Gamma(N\ms)$ of the normal bundle of $\ms$ where $Q^\ms$ is negative definite.
\end{definition}

Under the above assumptions, the Morse index of $\ms$ is a finite number equal to the number of negative eigenvalues of the following elliptic problem on $\Gamma(N\ms)$:
\[
\begin{cases}
\lapl_\ms^\perp Y + \Ric^\perp_\amb(Y,\cdot) +\abs {A^\ms}^2 Y + \lambda Y = 0 & \text{on $\ms$}\comma\\
\cov^\perp_\eta Y = - (\II^{\partial\amb}(Y,\cdot))^\sharp & \text{on $\partial\ms$}\point
\end{cases}
\]

\begin{definition}
We say that $\ms^n\subset M^{n+1}$ is \emph{stable} in an open subset $U\subset M$ if $Q^\ms(Y,Y)\ge 0$ for all $Y\in \Gamma(N\ms)$ compactly supported in $U$.
In particular, $\ms$ is stable in $M$ if and only if $\ind(\ms)=0$ 
\end{definition}

\subsection{Monotonicity formula and a density estimate}

In this section, we recall the monotonicity formula for free boundary minimal hypersurfaces. 

\begin{theorem} [{Monotonicity formula, cf. \cite{Simon1983}*{Section 17} and \cite{GuangLiZhou2020}*{Theorem~3.4}}] \label{thm:monotonicity}
Let $\ms^n\subset M^{n+1}$ be a free boundary minimal hypersurface. Then there exist constants $C_{MF} = C_{MF}(M,g)>0$ and $r_{MF}=r_{MF}(M,g) >0$ such that
\[
r\mapsto e^{C_{MF}r}\Theta^\Sigma_g(x,r)
\]
is nondecreasing for $0<r\le r_{MF}$ if $x\in\partial \Sigma$ and for $0<r\le \min\{r_{MF},d(x,\partial \Sigma)\}$ if $x\in\Sigma\setminus\partial \Sigma$. Here, $\Theta_g^\Sigma(x,r)$ denotes the density of $\Sigma$ centered at $x$ at radius $r$, namely
\[
\Theta^\Sigma_g(x,r)\eqdef \frac{\Haus^n(\ms\cap B(x,r))}{\omega_n r^n}.
\]

Moreover, if $M^{n+1}$ is $\R^{n+1}$ or $\R^{n+1}_+$ (with the Euclidean metric) and $\Theta^\Sigma_{\mathrm{Eucl}}(x,r) = \Theta^\Sigma_{\mathrm{Eucl}}(x,s)$ for some 
$x\in M$ and $0<s<r$, then $\Sigma$ is a free boundary minimal cone tipped at $x$.
\end{theorem}

Let us also recall a consequence of the monotonicity formula (see also \cite{AmbrozioCarlottoSharp2018Compactness}*{Corollary~16}), which will be useful in what follows.
\begin{corollary}\label{cor:density-estimate}
    Let $\ms^n\subset M^{n+1}$ be a free boundary minimal hypersurface. Then, there exist constants $r_{DE} = r_{DE}(M,g)>0$ and $C_{DE} = C_{DE}(M,g)>0$ such that, for every $x\in\ms$ and $0<r\leq r_{DE}$, it holds
    \[
        C_{DE}^{-1}
            \le \frac{\Haus^n(\ms\cap B(x,r))}{
                \omega_n r^n}
            \le C_{DE} \Haus^n(\ms).
    \]
\end{corollary}
\begin{proof}
    The proof is essentially a corollary of the monotonicity formula \cref{thm:monotonicity}. Observe that the behaviour of the ``corrected'' density
    \[
        e^{C_{MF} r}\frac{\Haus^n(\ms\cap B(x,r))}{\omega_n r^n}
    \]
   as $r\to 0$ changes depending on whether $x\in\ms$ is an interior or a boundary point: it approaches $1$ or $1/2$, respectively.
    We let $0< r \leq  r_{DE} \leq \frac{1}{3}r_{MF}$ for the rest of the proof.
    
    For the lower bound, note that, if $B(x,r)\cap\partial\ms$ is empty then we can use the interior monotonicity formula to get $  e^{-C_{MF} r} > e^{-C_{MF} r_{DE}}  $ as lower density bound. Similarly, if $x\in\partial\ms$ we can use the boundary monotonicity formula to get a lower density bound of $ e^{-C_{MF} r_{DE}}/2$. So we just need to study the case when $x\in\ms\setminus\partial\ms$ but $B(x,r)\cap \partial\ms\not=\emptyset$.
        
    In this case, depending on whether $d(x,\partial\ms)$ is less than or greater or equal to $r/2$ we have that either the ball $B(x,r/2)$ does not intersect the boundary or the ball $B(\bar{x},r/2)$ is contained in $B(x,r)$, where $\bar{x}$ is a nearest point to $x$ in $\partial \ms$, that is, $\bar{x}\in\partial\ms$ achieves $d(x,\partial\ms)$.
    Thus, by running the interior or the boundary monotonicity formulas in the corresponding smaller ball, we get $e^{-C_{MF} r_{DE}}2^{-(n+1)}$ as lower bound. 
        
    For the upper bound note that, if $x\in \partial\ms$ it suffices to use the boundary monotonicity formula up to $r=r_{DE}$ so that we get
        \[
            \frac{\Haus^n(\ms\cap B(x,r))}{\omega_n r^n}
                \leq \frac{e^{C_{MF}r_{DE}}}{\omega_n r_{DE}^n} \Haus^n(\ms)
                .
        \]
        
        If $x\in \ms\setminus\partial\ms$ then we can reduce to the case where $r\geq d(x,\partial\ms)/2$ simply by running the interior monotonicity formula if $r<d(x,\partial\ms)/2$. In this situation we have that $B(\bar{x},3r)\supseteq B(x,r)$, where $\bar x\in\partial\ms$ is a point achieving $d(x,\partial\ms)$.
        Therefore, using the boundary monotonicity formula we get
        \[
           \frac{\Haus^n(\ms\cap B(x,r))}{\omega_n r^n}
                \leq \frac{e^{C_{MF}3r_{DE}}}{\omega_n r_{DE}^n} \Haus^n(\ms). 
        \]
        So it is clear how big $C_{DE}$ has to be. Notice that this, and the choice of $r_{DE}$, only depend on $C_{MF}$ and $r_{MF}$, so the estimate is proven.
\end{proof}
\begin{remark}\label{rem:how-to-use-density-estimate}
    We are also interested in applying the \emph{lower bound} of this density estimate for hypersurfaces of the type $S = \ms\cap B(p,s)$, for some $p\in M$, $s>0$.
    In this case, we need to make sure that $B(x,r)\subset B(p,s)$, in which case we can apply the above proof with no changes replacing $\ms$ with $S$.
\end{remark}

\subsection{Curvature estimates}

We finally recall the curvature estimates for stable free boundary minimal hypersurfaces in ambient manifolds of dimension $3\le n+1\le 7$.

\begin{theorem}[{Curvature estimates}] \label{thm:curvature-estimates}
Assume that the dimension of the ambient manifold $\amb^{n+1}$ is $3\le n+1\le 7$, and let $\CA>0$ be a constant. 

Then, there exists a constant $C_{CE} = C_{CE}(M,g,\CA) > 0$ so that, for every free boundary minimal hypersurface $\ms^n\subset M$ that is stable in $B(p,r)$, for some $p\in M$ and $0<r<\conv_M$, and with $\Haus^n(\ms)\leq \CA$, it holds that
    \[
		\sup_{x\in\ms\cap B(p,r/2)} \abs{A^\ms}(x)\leq \frac{C_{CE}}{r}.
	\]
	Moreover, if $n=2,3$ the constant $C_{CE}=C_{CE}(M,g)$ does not depend on $\CA$.
\end{theorem}
\begin{proof}
First note that, since $0<r<\conv_M$, $B(p,r)$ is simply connected and thus $\ms\cap B(p,r)$ is two-sided (cf. \cite{Samelson1969},  \cite{ChodoshKetoverMaximo2017}*{Appendix C}). Hence, the result follows from \cite{GuangLiZhou2020}*{Theorem~1.1}.
If $n=2$, the fact that $C_{CE}$ does not depend on the area bound $\CA$ is proved in \cite{GuangLiZhou2020}*{Theorem~1.2}, using \cite{Schoen1983Estimates}*{Theorem~3}. One can similarly prove that the constant $C_{CE}$ does not depend on $\CA$ for $n=3$ using \cite{ChodoshLi2021}*{Theorem~3}.
\end{proof}

\section{Stability balls} \label{sec:stability}

In this section, let us assume again that $(M^{n+1},g)$ is a compact Riemannian manifold of dimension $3\le n+1\le 7$ with (possibly empty) boundary. Moreover, let us assume that $\ms^n\subset \amb$ is a smooth, compact, properly embedded, free boundary minimal hypersurface.

\subsection{Stability radius}
Here we recall the definition of stability radius with a couple of properties.

\begin{definition}
\label{def:stability-radius}
	Let $\lambda\ge 2$ and $\bar{r} >0$ be given parameters.
	For all $p\in M$, we define the \emph{stability radius} associated to $\ms$ at $p$ as 
	\[
    	s_\ms(p) = s_\ms(p,\lambda,\bar{r})
    	:= \sup\{ r\leq\bar{r} \st  \ms \text{ is stable in }B(p,\lambda r)\},
	\]
	where, by convention, we assume that $\ms$ is stable in $B(p,\lambda r)$ in the case $\ms\cap B(p,\lambda r)= \emptyset$.
  We say that $B\subset M$ is a \emph{stability ball} if $B=B(p,s_\ms(p))$ for some $p\in M$.
\end{definition}

\begin{lemma} \label{lem:stab-continuous-positive}
    The stability radius is $(1/\lambda)$-Lipschitz continuous and is strictly positive at every point. In particular, $s_\ms$ attains a strictly positive minimum on $M$.
\end{lemma}
\begin{proof}
    Assume $\ms\cap B(p,\lambda s)$ is stable, then $\ms \cap B(q,\lambda s- \dist(p,q))$ is also stable for all $q\in \amb$, since it is contained in $\ms \cap B(p,\lambda s)$ and stability is a property that is inherited by open subregions. Therefore, we have that $s_\ms(q) \geq s- \frac{\dist(p,q)}{\lambda}$. Thus, taking the supremum on $s\leq \bar{r}$ so that $\ms \cap B(p,\lambda s)$ is stable, we see that
    \[
        s_\ms(p) - s_\ms(q) \leq \frac{1}{\lambda} \dist(p,q).
    \]
    By symmetry, we get that $s_\ms$ is $(1/\lambda)$-Lipschitz.
    
    The fact that $s_\ms$ is strictly positive at every point also follows from standard arguments (see e.g. \cite{Song2019}*{Lemma 20}). Therefore, by continuity and compactness of $M$, $s_\ms$ attains a strictly positive minimum.
\end{proof}

{
\begin{lemma} \label{cor:BetterBalls}
Assume that $\lambda\ge 3$, then there exists $\bar r = \bar r(M,g)>0$ small enough in \cref{def:stability-radius} such that every point in $M$ is contained either in a stability ball $B(p,s_\ms(p))$ that is at positive distance from $\partial M$ or in a stability ball $B(p,s_\ms(p))$ centered at a point $p\in\partial M$.
\end{lemma}
\begin{proof}
Fix any $q\in M$. If $q\in\partial M$ or $B(q,s_\ms(q))$ is at positive distance from $\partial M$, then we are done. Therefore, assume that neither of the previous two cases holds. 
First observe that, for all $q'\in B(q,\mu s_\ms(q))$ where $\mu\eqdef (1+\lambda^{-1})^{-1}=\lambda/(\lambda+1)$, by the previous lemma it holds that
\[
s_\ms(q') \ge s_\ms(q) - \frac1\lambda d(q,q') > \left(1+\frac 1\lambda\right)d(q,q') - \frac1\lambda d(q,q') \ge d(q,q').
\]
Therefore, we get that $q\in B(q',s_\ms(q'))$.

Let $p_1$ be the point on $\partial M$ realizing the distance $d(q,\partial M)$, i.e. $d(q,p_1) = d(q,\partial M)\le s_\ms(q)$. If $d(p_1,q)< \mu s_\ms(q)$, then we are done by the previous observation. Therefore, assume that $d(p_1,q)\ge \mu s_\ms(q)$.
Now, taking $\bar r=\bar r(M,g)>0$ sufficiently small, we can find a point $p_2\in B(q,\mu s_\ms(q))$ (e.g. moving along the geodesic emanating from $p_1$ and going through $q$) such that 
\[
d(p_2,\partial M) = d(p_2,p_1) = d(p_2,q) + d(q,p_1) > \left(\frac{1}{\lambda-1} + 1\right)d(p_1,q) = \frac{\lambda}{\lambda-1}d(p_1,q).
\]
Note that this can easily be done for any $\lambda\ge 3$, because $d(p_1,q)\le s_\ms(q)< (\lambda-1)\mu s_\ms(q)$ and $p_2$ can be chosen such that $d(p_2,q)$ is arbitrarily close to $\mu s_\ms(q)$. At this point, we have
\begin{itemize}[wide]
    \item either $B(p_2,s_\ms(p_2))$ is at positive distance from $\partial M$, and contains $q$;
    \item or $d(p_1,p_2) \le s_\ms(p_2)$, in which case we have
    \begin{align*}
    s_\ms(p_1) &\ge s_\ms(p_2) -\frac 1\lambda d(p_1,p_2) \ge\left(1-\frac1\lambda\right) d(p_1,p_2)> \left(1-\frac1\lambda\right) \frac{\lambda}{\lambda-1} d(p_1,q) = d(p_1,q),
    \end{align*}
    which proves that $q\in B(p_1,s_\ms(p_1))$ and concludes the proof since $p_1\in \partial M$.\qedhere
\end{itemize}
\end{proof}

\begin{definition} \label{def:GoodBalls}
We denote by $\mathcal{B}_\ms$ the set of the stability balls $B(p,s_\ms(p))$ for $p\in M$ such that either $d(p,\partial M) > s_\ms(p)$ or $p\in\partial M$.
\end{definition}

\begin{remark}
Note that $\mathcal{B}_\ms$ is a cover of $M$ by \cref{cor:BetterBalls}. Therefore, in all our arguments we can work only with stability balls in $\mathcal{B}_\ms$, which are better behaving.
\end{remark}
}

The following lemma is a consequence of the curvature estimates \cref{thm:curvature-estimates} and says that the second fundamental form of a free boundary minimal hypersurface is controlled at the radius of stability. 
\begin{lemma} \label{lem:BoundAbove}
For any $\varepsilon>0$, there exist constants $\lambda = \lambda(\amb,g,C_{CE},\varepsilon)\geq 2$ large enough and $\bar{r}=\bar{r}(\amb,g,C_{CE},\varepsilon)>0$ small enough in \cref{def:stability-radius} such that, for all $p\in M$, it holds
\[
s_\ms(p) \sup_{x\in\ms\cap B(p,3s_\ms(p))} \abs{A^\ms}(x) \le \varepsilon.
\]
\end{lemma}
\begin{proof}
By the curvature estimates \cref{thm:curvature-estimates} we have that, for any $\varepsilon>0$, provided we choose $\lambda=\lambda(M,g,C_{CE},\varepsilon)\ge 6$ large enough (and making $\bar{r} = \bar{r}(M,g,\lambda)>0$ small so that $\lambda\bar{r}<\conv_M$) we get
    \[
        \sup_{x\in\ms\cap B(p,3s_\ms(p))} \abs{A^\ms}(x)     \le \frac{C_{CE}}{\lambda s_\ms(p)}
            \leq \frac{\varepsilon}{s_\ms(p)}. \qedhere
    \]
\end{proof}

On the other hand, the following lemma shows that, passed the radius of stability, there exists one point where the second fundamental for of the free boundary minimal hypersurface is sufficiently large.
\begin{lemma} \label{lem:BoundBelow}
 We can choose $\lambda = \lambda(\amb,g)\geq 2$ large enough and $\bar{r}=\bar{r}(\amb,g)>0$ small enough in \cref{def:stability-radius} such that, for all $p\in M$ with $s_\ms(p)<\bar r$, there exists $q\in B(p, 2\lambda s_\ms(p))$ with
 \[
s_\ms(p) \abs{A^\ms}(q) > \kappa,
 \]
 for some constant $\kappa=\kappa(M,g,\lambda)>0$.
\end{lemma}
\begin{proof}
This is the analogue in the free boundary case of \cite{Song2019}*{(8)}, but the proof is the same. Let us sketch it for completeness. 
Taking $\bar r$ possibly smaller (depending on $\lambda$), we can assume that the metric in $B(p,2\lambda\bar r)$ is arbitrary close to the Euclidean metric. 
Let us now consider a constant $\kappa>0$ to be fixed later. If $s_\ms(p)\abs{A^\ms}(q)\le \kappa$ for all $q\in B(p,2\lambda s_\ms(p))$, then, choosing $\kappa$ sufficiently small (depending on $\lambda$), we have that (in chart) $\ms$ is arbitrary close (in the $C^\infty$-graph sense) to a hyperplane in a Euclidean ball. In particular, $\ms\cap B(p,2\lambda s_\ms(p))$ is stable.
This proves that, if $s_\ms(p)<\bar r$, which implies that $\ms\cap B(p,2\lambda s_\ms(p))$ is unstable, then there exists $q\in B(p,2\lambda s_\ms(p))$ such that $s_\ms(p)\abs{A^\ms}(q)>\kappa$, concluding the proof.
\end{proof}

\subsection{Controlled topology at the radius of stability}

In this subsection, we prove that (suitably choosing $\lambda,\bar r$ in the definition \cref{def:stability-radius} of the stability radius) the topology of a minimal hypersurface is well-controlled in stability balls, which follows from the curvature estimates \cref{thm:curvature-estimates} for stable free boundary minimal hypersurfaces, or more precisely from its consequence \cref{lem:BoundAbove} mentioned above.

\begin{proposition} \label{prop:write-as-graph}
    We can choose $\lambda = \lambda(\amb,g,C_{CE})\geq 2$ large enough and $\bar{r}= \bar{r}(\amb,g,C_{CE})>0$ small enough in the definition of the stability radius $s_\ms(p) = s_\ms(p,\lambda,\bar r)$ (cf. \cref{def:stability-radius}), so that the following {properties hold}.

    \begin{enumerate}[label={\normalfont(\roman*)}]
    \item\label{wag:i} For all $B=B(p,s_\ms(p))\in \mathcal{B}_\ms$, each component of $\ms\cap B(p,3s_\ms(p))$ either does not intersect $B$ or its intersection with $B$ consists of only one connected component.
    \item\label{wag:ii} For all $B\in\mathcal{B}_\ms$, we have
     \begin{equation*}
    	\abs{\left\{\text{components of } \partial \ms\cap B \right\}} \leq \abs{\left\{\text{components of } \ms\cap B\right\}}.
    \end{equation*}
    \item\label{wag:iii} For every $B_1,\ldots,B_N\in \mathcal{B}_\ms$, we have that $\Sigma \cap\bigcap_{i=1}^NB_i$ and $\partial\Sigma \cap\bigcap_{i=1}^NB_i$ are either empty or topological discs (of dimension $n$ and $n-1$, respectively).
    \end{enumerate}
\end{proposition}

\begin{proof}
The proof of this proposition is fairly standard, so we only give the main ingredients for completeness.
First note that, if the ball $B\in\mathcal{B}_\ms$ or at least one of the balls $B_1,\ldots,B_N\in\mathcal{B}_\ms$ are contained in the interior of $M$, then the result follows from the closed case in \cite{Song2019}*{Section~2.1}.
Therefore, let us assume that all the balls in the statement are centered at the boundary (recall that, by \cref{def:GoodBalls}, this is the only other possibility).

Let us then consider a ball $B=B(p,s_\ms(p))\in\mathcal{B}_\ms$ such that $p\in\partial M$. 
We can assume (taking $\bar r=\bar r(M,g)>0$ sufficiently small) to be working in a chart such that $B(p,3\bar r)$ is sent to the unit half-ball in $\R^{n+1}_+\eqdef \R^{n+1}\cap\{x_{1}\ge 0\}$, the orthogonality condition along $\{x_1=0\}$ with respect to $g$ and to the Euclidean metric are the same, and the metric in this chart is sufficiently close to the Euclidean one (in the $C^3$ topology).
This can be achieved by using Fermi coordinates.

In particular, we can assume that every ball $B'=B(q,r)$ contained in $B(p,3\bar r)$ with $q\in\partial M$ is sufficiently close to a half-ball in the sense that:
\begin{itemize}
\item the Euclidean second fundamental form of $\partial B'\setminus\partial M$ is positive (in the sense that, applied to any tangent vector to $\partial B'$, it points outside of $B'$) with norm bounded below by $1/(6s_\ms(p))$;
\item $B'\cap\partial M$ is a topological $(n-1)$-disc contained in the plane $\{x_1=0\}$ (where $\partial M$ is mapped) and the second fundamental form of $\partial B' \cap \partial M$ in $\{x_1=0\}$ is positive with norm bounded below by $1/(6s_\ms(p))$.
\end{itemize}
Note that here we are using that $q\in\partial M$, otherwise for example $B'$ could have more complicated intersection with $\partial M$.
    
Thanks to \cref{lem:BoundAbove}, for any $\varepsilon>0$, provided we choose $\lambda=\lambda(M,g,C_{CE},\varepsilon)\geq 6$ large enough and $\bar{r} = \bar{r}(M,g,C_{CE},\varepsilon)>0$ small enough, we get
    \[
        s_\ms(p)\sup_{x\in\ms\cap B(p,3s_\ms(p))} \abs{A^\ms}(x)     \le \varepsilon.
    \]
Thus, given that the corresponding Euclidean quantities are comparable and choosing $\varepsilon$ (universally) small enough, we can apply \cite{ColdingMinicozzi2011}*{Lemma~2.4} to obtain that any connected component $\tilde\ms$ of $\ms$ in $B(p,3s_\ms(p))$ is a graph with bounded gradient over the half-hyperplane $\R^n_+\eqdef\{x_{n+1}=0,x_1\ge0\}$ (thanks to the free boundary condition).

Furthermore, since $B$ is close to the corresponding Euclidean half-ball as described above, $\varepsilon>0$ can be chosen small enough such that the Euclidean second fundamental form of $\tilde\ms$ is strictly less than the one of $\partial B\setminus\partial M$. Therefore, $\tilde\ms\cap B$ is the graph over a topological $n$-disc.

The same argument proves that $\tilde\Sigma\cap\partial M$ is a topological $(n-1)$-disc, since (with the same choice of $\varepsilon$) the Euclidean second fundamental form of $\tilde\ms\cap\partial M$ is strictly less than the one of $\partial B\cap \partial M$ (recall that $\ms$ is orthogonal to $\{x_1=0\}$ by the choice of Fermi chart). 

From this observation, the proof of the proposition follows easily. Indeed, we have that $\tilde\ms\cap B$ and $\tilde\ms\cap \partial B$ are connected (proving \ref{wag:i} and \ref{wag:ii}).
Moreover, choosing $B=B_1$ and possibly other stability balls $B_2,\ldots, B_N\in\mathcal{B}_\ms$ centered at the boundary $\partial M$, we can repeat the same argument for each ball (since they all have Euclidean second fundamental form bounded below by $1/(6s_\ms(p))$) and get that $\tilde{\ms}\cap (B_1\cap\ldots\cap B_N)$ and $\partial\tilde\ms\cap (B_1\cap\ldots\cap B_N)$ are topological discs, which proves \ref{wag:iii}.
\end{proof}

Thanks to the previous proposition about the behaviour of $\ms$ on a ball of stability, we obtain the following two corollaries.

\begin{corollary} \label{cor:acyclic-cover}
    We can choose $\lambda = \lambda(\amb,g,C_{CE})\geq 2$ large enough and $\bar{r}=\bar{r}(\amb,g,C_{CE})>0$ small enough in \cref{def:stability-radius} so that {
        \[
    	    \left\{  \text{components of } \ms\cap B \st  B\in\mathcal{B}_\ms  \right\}
    	    \,\text{ is an acyclic cover of }\ms,
    	\]
    	and
    	\[
    	    \left\{ \text{components of } \partial \ms\cap B\st  B\in\mathcal{B}_\ms  \right\}
    	    \,\text{ is an acyclic cover of }\partial \ms.
    	\]}
\end{corollary}
\begin{proof}
    Let $\lambda = \lambda(\amb,g,C_{CE})\geq 2$ and $\bar r = \bar r(\amb,g,C_{CE})>0$ be such that \cref{prop:write-as-graph} holds.
    {Then, as a consequence of the proposition, for all $B_1,\ldots,B_N\in\mathcal{B}_\ms$, we have that $\ms\cap \bigcap_{i=1}^NB_i$ and $\partial\ms\cap \bigcap_{i=1}^NB_i$ are either empty or have trivial cohomology, which implies that the two covers in the statement are acyclic.}
\end{proof}

\subsection{Bounding the number of sheets in balls of stability}

Together with the area bound, the consequence of the monotonicity formula, \cref{cor:density-estimate}, implies that $\ms$ cannot have arbitrarily many sheets in the balls of the cover obtained in \cref{cor:acyclic-cover}.
\begin{proposition} \label{prop:sheeting-results}
    There exist $\lambda = \lambda(\amb,g,C_{CE})\geq 2$ large enough, $\bar{r}=\bar{r}(\amb,g,C_{CE})>0$ small enough in \cref{def:stability-radius}, and a constant $C_{CS}= C_{CS}(\amb,g) > 0$ such that
    {
    \begin{equation*} \label{eq:sheeting-for-sigma}	
		\abs{\left\{\text{components of } \ms\cap B \right\}}
		\leq C_{CS}\CA,
	\end{equation*}
    for all $B\in\mathcal{B}_\ms$.}
\end{proposition}
\begin{proof}
    Let $p\in M$ be the center of $B$ such that $B=B(p,s_\ms(p))$. Choose $\lambda=\lambda(M,g,C_{CE})\geq 2$ and $\bar r =\bar r(M,g,C_{CE})>0$ so that \cref{prop:write-as-graph} holds. 
    Let $S_1,\ldots,S_N$ be connected components of $\ms\cap B(p,3s_\ms(p))$ that intersect $B(p,s_\ms(p))$. By~\cref{prop:write-as-graph}\ref{wag:i}, each $S_i$ for $i=1,\ldots,N$ intersects $B(p,s_\ms(p))$ in a single connected component; thus, it suffices to obtain a bound for $N$.
    
    We let $x_1,\ldots,x_N$ be such that $x_i\in S_i\cap B(p,s_\ms(p))$ for all $i=1,\ldots,N$. Note that $B(x_i,s_\ms(p))\subset B(p,2s_\ms(p))$.
    Hence, choosing $2\bar r<r_{DE}$, we can apply the lower density estimate \cref{cor:density-estimate} to each of the $S_i$ (see \cref{rem:how-to-use-density-estimate}) and the upper density estimate to $\ms$, obtaining
    \begin{align*}
        N
            &= C_{DE}\sum_{i=1}^N C_{DE}^{-1}
            \leq C_{DE}\sum_{i=1}^N  \frac{\Haus^n(S_i\cap B(x_i,s_\ms(p))}{\omega_n s_{\ms}(p)^n}
            \\
            &\leq C_{DE} 2^n\frac{\Haus^n(\ms\cap B(p,2s_\ms(p))}{\omega_n (2s_{\ms}(p))^n}
            \leq 2^n C_{DE}^2 \Haus^n(\ms)
            \leq  C_{CS}\CA,
    \end{align*}
    where we used that $S_i\cap B(x_i,s_\ms(p))$ are disjoint subregions of $\ms\cap B(p,2s_\ms(p))$ and where we set $C_{CS} \eqdef 2^n C_{DE}^2$.
\end{proof}

\section{Setting and choice of parameters} \label{sec:choice-param}

From now on, if not otherwise stated, $(M^{n+1},g)$ denotes a compact Riemannian manifold of dimension $3\le n+1\le 7$ with (possibly empty) boundary and $\CA>0$ is a positive constant.
We let $\ms^n\subset \amb$ be a smooth, compact, properly embedded, free boundary minimal hypersurface with $n$-volume less or equal than $\CA$.
Moreover, we assume that the stability radius $s_\ms$ is defined in \cref{def:stability-radius} with respect to parameters $\lambda=\lambda(M,g,C_{CE})\ge 2$ and $0<\bar r = \bar r(M,g,C_{CE}) <1$ such that
\begin{itemize}
\item \cref{lem:Besicovitch-type-lemma} and \cref{prop:we-can-layer} hold;
\item \cref{lem:BoundBelow,prop:write-as-graph,cor:acyclic-cover,prop:sheeting-results} hold;
\item the product $100\lambda\bar r$ is less than $\conv_M$, $r_{MF}$, $r_{DE}$, defined in \cref{def:convexity-radius,thm:monotonicity,cor:density-estimate} respectively.
\end{itemize}

\begin{case3}
Observe that we write explicitly the dependence on the constant $C_{CE}$, since $C_{CE}$ depends on $\CA$ only for $n+1\ge 5$. Therefore, actually, $\lambda$ and $\bar r$ depend only on {$(M,g)$} for $n+1=3,4$ and depend also on $\CA$ for $n+1=5,6,7$.
\end{case3}

Observe that, setting $\bar \CA =\bar \CA(M,g,\CA)\eqdef 2\omega_n C_{DE} \CA$, by \cref{cor:density-estimate} we have that \[\Haus^n(\Sigma\cap B(p,r))\le \bar \CA r^n/2\] for all $p\in M$ and $0<r\le 100\lambda\bar r < r_{DE}$.

\section{The almost conical region} \label{sec:conical-region}

In this section we define the almost conical region in the setting with boundary, following \cite{Song2019}*{Section 2.2}.
In the setting of \cref{sec:choice-param}, the idea is to define a region in $M$ where $\Sigma$ has well-controlled curvature and topology. This region looks like a union of annuli where $\Sigma$ is close to Euclidean cones.

\subsection{Choice of cones} \label{sec:choice-cones}
Let us start by identifying a set of cones that will play the role of models for our free boundary minimal hypersurface in the ``almost conical region'' that we are about to define. In this subsection we define these Euclidean cones we are interested in and we prove some properties.

Given constants $\bar \CA \ge 1$ and $\BO>0$, we define $\mathcal{G}=\mathcal{G}(n,\bar \CA,\BO)$ to be the family of minimal cones $\Gamma\subset(\R^{n+1}, g_{\mathrm{Eucl}})$ and (possibly nonproperly embedded) free boundary minimal cones $\Gamma\subset(\R^{n+1}_+, g_{\mathrm{Eucl}})$, which are tipped at $0$, smooth outside $0$, and which satisfy
\[\Theta^\Gamma_{\mathrm{Eucl}}(0,1)\le\bar \CA \quad \text{ and } \quad \max_{y\in\Gamma\cap \partial B^{n+1}_{\mathrm{Eucl}}(0,1)} \abs{A^\Gamma(y)} \le \BO.
\]

\begin{remark}
Observe that the only nonproperly embedded free boundary minimal cone in $\mathcal{G}$ is the boundary plane $\partial\R^{n+1}_+$ in $\R^{n+1}_+$.
\end{remark}
\begin{remark}
Observe that the definition is the analogue in the free boundary setting to the definition of $\mathcal{G}_{\beta_0}$ in \cite{Song2019}*{Section 2.2}.
\end{remark}

\begin{case3}
Note that, when $n+1=3$, the only minimal cones in $(\R^{n+1},g_{\mathrm{Eucl}})$ and free boundary minimal cones in $(\R^{n+1}_+,g_{\mathrm{Eucl}})$ tipped at $0$ and smooth outside of $0$ are planes or half-planes. As a result, when $n=2$, $\mathcal{G} = \mathcal{G}(2)$ does not depend on $\bar \CA$ and $\BO$.
\end{case3}

\begin{lemma} \label{lem:TwoBallsBadCones}
There exists $\BO=\BO(n,\bar\CA)>0$ sufficiently large such that the following property holds. Let $\Gamma$ be a minimal cone in $(\R^{n+1},g_{\mathrm{Eucl}})$ or a free boundary minimal cone in $(\R^{n+1}_+,g_{\mathrm{Eucl}})$, which is tipped at $0$, smooth outside of $0$, and such that \[\Theta^\Gamma_{\mathrm{Eucl}}(0,1)\le\bar \CA \quad \text{ and } \quad \max_{y\in\Gamma\cap \partial B^{n+1}_{\mathrm{Eucl}}(0,1)} \abs{A^\Gamma(y)} > \BO.
\]
Then, there exist two stability balls $b,b'\subset A_{\mathrm{Eucl}}(0,1/2,2)$ such that \[
4\lambda b\cap 4\lambda b'=\emptyset \quad  \text{ and } \quad \radius(b),\radius(b') < 2^{-2000}.
\]

\begin{case3}
Observe that, as remarked above, there are no such cones for $n+1=3$. Therefore the lemma is trivially true.
\end{case3}
\end{lemma}
\begin{proof}
Assume that $s_\Gamma(p)\ge  3\cdot 2^{-2002}$ for all $p\in A_{\mathrm{Eucl}}(0,3/4,1)$, then by \cref{thm:curvature-estimates} the second fundamental form of $\Gamma$ is bounded above by a constant depending on $n$ and $\bar \CA$ (and depending only on $n$ for $n+1\le4$). Therefore, setting $\BO=\BO(n,\bar\CA)$ sufficiently large, we can assume that there exists $p\in A_{\mathrm{Eucl}}(0,3/4,1)$ such that $s_\Gamma(p)< 3 \cdot 2^{-2002}$. Then, we can choose $b\eqdef B^{n+1}_{\mathrm{Eucl}}(p,s_\Gamma(p))$ and $b' \eqdef B^{n+1}_{\mathrm{Eucl}}(4p/3,s_\Gamma(4p/3))$ to conclude the proof. Note that $s_\Gamma(4p/3) = 4s_\Gamma(p)/3$.
\end{proof}

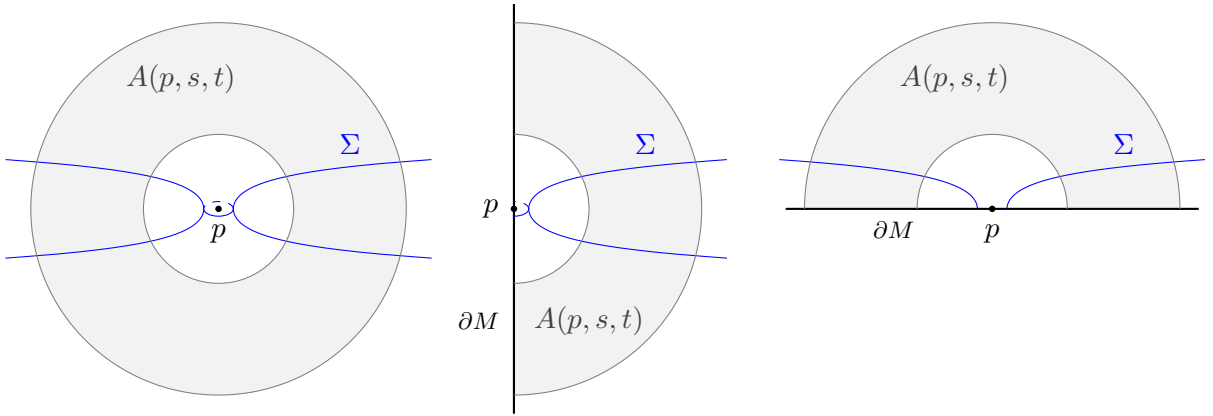
\begin{figure}[htpb]
\centering
\begin{tikzpicture}[scale=1.85]
\pgfmathsetmacro{\xa}{-3.2}
\pgfmathsetmacro{\xb}{-1.1}
\pgfmathsetmacro{\xc}{2.3}
\pgfmathsetmacro{\s}{0.8}
\pgfmathsetmacro{\t}{0.3}

\pgfmathsetmacro{\R}{0.45}
\pgfmathsetmacro{\f}{1.05}
\pgfmathsetmacro{\T}{0.8}
\pgfmathsetmacro{\a}{1/(\T*cosh(\T))}

\pgfmathsetmacro{\l}{0.25}
\pgfmathsetmacro{\ra}{\l*\R*\a*cosh(\T/\l)}
\pgfmathsetmacro{\rb}{\R*\a*\T}
\pgfmathsetmacro{\r}{sqrt(\ra*\ra+\rb*\rb)}

\fill[gray!10!white] (\xa,0) circle (\r);
\fill[white] (\xa,0) circle (0.4*\r);

\fill[gray!10!white] (\xb,\r) arc (90:-90:\r);
\fill[white] (\xb,{0.4*\r}) arc (90:-90:{0.4*\r});

\fill[gray!10!white] (\xc+\r,0) arc (0:180:\r);
\fill[white] (\xc+0.4*\r,0) arc (0:180:0.4*\r);

\begin{scope}[myBlue]
\draw[scale=1,domain=-\f*\T:\f*\T,smooth,variable=\s]  plot ({\xa+\l*\R*\a*cosh(\s/\l)},{\R*\a*\s});
\draw[scale=1,domain=-\f*\T:\f*\T,smooth,variable=\s]  plot ({\xa-\l*\R*\a*cosh(\s/\l)},{\R*\a*\s});
\coordinate(P)at({\xa+\l*\R*\a},{0});
\draw[dashed] (P) arc (0:180:{\l*\R*\a} and {0.5*\l*\R*\a});
\draw (P) arc (0:-180:{\l*\R*\a} and {0.5*\l*\R*\a});

\end{scope}

\draw[gray] (\xa,0) circle (\r);
\draw[gray] (\xa,0) circle (0.4*\r);

\node[blue] at (\xa+0.7*\r,0.35*\r) {$\ms$};
\fill (\xa, 0) circle [radius=0.06em] node[below=2pt] {$p$};
\node[gray!50!black] at (\xa-0.2*\r,0.7*\r) {$A(p,s,t)$};

\draw[thick] ({\xb},1.1*\r) -- ({\xb},-1.1*\r);

\begin{scope}[myBlue]
\draw[scale=1,domain=-\f*\T:\f*\T,smooth,variable=\s]  plot ({\xb+\l*\R*\a*cosh(\s/\l)},{\R*\a*\s});
\coordinate(P)at({\xb+\l*\R*\a},{0});
\draw[dashed] (P) arc (0:90:{\l*\R*\a} and {0.5*\l*\R*\a});
\draw (P) arc (0:-90:{\l*\R*\a} and {0.5*\l*\R*\a});
\end{scope}

\draw[gray] (\xb,\r) arc (90:-90:\r);
\draw[gray] (\xb,{0.4*\r}) arc (90:-90:{0.4*\r});

\node[blue] at (\xb+0.7*\r,0.35*\r) {$\ms$};
\fill (\xb, 0) circle [radius=0.06em] node[left=2pt] {$p$};
\node[gray!50!black] at (\xb+0.4*\r,-0.6*\r) {$A(p,s,t)$};

\node at ({\xb-0.25}, -\s) {\footnotesize $\partial \amb$};

\draw[thick] ({\xc-1.1*\r},0) -- ({\xc+1.1*\r},0);

\begin{scope}[myBlue]
\draw[scale=1,domain=0:\f*\T,smooth,variable=\s]  plot ({\xc-\l*\R*\a*cosh(\s/\l)},{\R*\a*\s});
\draw[scale=1,domain=0:\f*\T,smooth,variable=\s]  plot ({\xc+\l*\R*\a*cosh(\s/\l)},{\R*\a*\s});
\end{scope}

\draw[gray] (\xc+\r,0) arc (0:180:\r);
\draw[gray] (\xc+0.4*\r,0) arc (0:180:0.4*\r);

\node at ({\xc-0.7}, -0.15) {\footnotesize $\partial\amb$};

\node[blue] at (\xc+0.7*\r,0.35*\r) {$\ms$};
\fill (\xc, 0) circle [radius=0.06em] node[below=2pt] {$p$};
\node[gray!50!black] at (\xc-0.2*\r,0.7*\r) {$A(p,s,t)$};
\end{tikzpicture}
\caption{Three different prototypical cases of an annulus $\delta$-close to a cone (gray region in figure).} \label{fig:DefConicalAnnuli}
\end{figure}

\begin{definition}
Given a free boundary minimal hypersurface $\Sigma^n\subset (M^{n+1},g)$, $p\in M$ and $0<s\le t/2$, we say that $\Sigma\cap A(p,s,t)$ is \emph{$\delta$-close to a cone in $\mathcal{G}$} for some $\delta>0$ (see also \cref{fig:DefConicalAnnuli}) if: either $B(p,t)\cap \partial M=\emptyset$ and $\Sigma\cap A(p,s,t)$ is $\delta$-close to a cone in $\mathcal{G}$ as defined in \cite{Song2019}*{Section~2.2.3}, or $p\in\partial M$ and the following property holds.
For all $r\in[s,t/2]$ and for each component $\Sigma'$ of $\Sigma\cap A(p,r,2r)$, there is a diffeomorphism $\Phi\colon B(p,2r)\subset (M,g)\to B^{n+1}_{\mathrm{Eucl}}(0,2)\cap \R^{n+1}_+$ such that
\begin{itemize}
\item $\Phi^*(g)$ is $\delta$-close to $g_{\mathrm{Eucl}}$ in the $C^5$-topology;
\item $\Phi(A(p,r,2r)) = A_{\mathrm{Eucl}}(0,1,2)\cap \R^{n+1}_+$;
\item there is a (possibly nonproperly embedded) free boundary minimal cone $\Gamma\in\mathcal{G}$ such that $\Phi(\Sigma') = \Gamma\cap A_{\mathrm{Eucl}}(0,1,2)$.
\end{itemize}
\end{definition}

The following analogue of \cite{Song2019}*{Theorem 8} at the boundary holds.

\begin{theorem} \label{thm:BoundaryNonConcentration}
Let $3\le n+1\le 7$, and let $\bar\CA,\lambda,\bar r, \delta, K$ be constants, and set $\bar K \eqdef 30\lambda K^2$. Moreover, assume that $\bar r\ge 1$ and $\BO=\BO(n,\bar\CA)$ is given by \cref{lem:TwoBallsBadCones}.
Then, there exist $\BU=\BU(n,\bar\CA,\lambda, \delta, K)>0$, $\varepsilon=\varepsilon(n,\bar \CA,\lambda,\delta, K)>0$ and $\bar R = \bar R(n,\bar \CA,\lambda, \delta, K)>1000$ such that the following statement holds.

Let $g$ be a metric on $B_{\mathrm{Eucl}}^{n+1}(0,\bar K)\subset \R^{n+1}$ that is $\varepsilon$-close to the Euclidean metric in the $C^5$-topology. 
Let $\Sigma^n$ be a minimal hypersurface in $B_{\mathrm{Eucl}}^{n+1}(0,\bar K)$ or a free boundary minimal hypersurface in $B_{\mathrm{Eucl}}^{n+1}(0,\bar K)\cap \R^{n+1}_+$, with respect to $g$, such that $\partial\Sigma\setminus \partial \R^{n+1}_+ \subset \partial B_{\mathrm{Eucl}}^{n+1}(0,\bar K)$. Assume that 
\begin{enumerate}[label=\normalfont(\roman*)]
\item\label{bnc:ii} $s_\Sigma(0) \le 1/\bar K$;
\item\label{bnc:iii} there is $y'\in B(0,7\lambda)$ with $s_\Sigma(y') = 1$;
\item\label{bnc:i} the $n$-volume of $\Sigma$ is at most $\bar \CA\bar K^n/2$;
\item\label{bnc:iv} $\Theta_g^\Sigma(0,10\lambda K) - \Theta_g^\Sigma(0,1/(3K))\le\BU$.
\end{enumerate}
Then, we have that
\begin{enumerate} [label=\normalfont(\arabic*)]
\item \label{c1:pair-of-balls} either there are two stability balls $b,\, b'\subset B(0,\bar K)$ with the property that $3\lambda b\cap 3\lambda b'=\emptyset$ and $\radius(b),\radius(b') \in [2^{-(R+1)},2^{-R})$ for some $R\in[1000,\bar R]$;
\item \label{c2:conic-annulus} or $\Sigma\cap A(0,1/(2K), 7\lambda K)$ is $\delta$-close to a cone in $\mathcal{G}=\mathcal{G}(n,\bar \CA,\BO)$.
\end{enumerate}

\begin{case3}
When $n+1=3$, a stronger statement holds, analogously to \cite{Song2019}*{Theorem~14}. Roughly, it is not necessary to assume the bound on the area \ref{bnc:i} and the density estimate \ref{bnc:iv}. This is the main point where the proof for $n+1=3$ differs from the higher dimensional case, so let us state precisely the result.

\begin{theorem3}
Let $\lambda,\bar r,\delta, K$ be constants and assume that $\bar r\ge 1$.
Then, there exist $\bar K = \bar K(\lambda,\delta,K)\ge 30\lambda K^2$, $\varepsilon=\varepsilon(\lambda, \delta, K)>0$ and $\bar R = \bar R(\lambda, \delta, K)>1000$ such that the following statement holds.

Let $g$ be a metric on $B_{\mathrm{Eucl}}^{3}(0,\bar K)\subset \R^{3}$ that is $\varepsilon$-close to the Euclidean metric in the $C^5$-topology. 
Let $\Sigma^2$ be a minimal surface in $B_{\mathrm{Eucl}}^{3}(0,\bar K)$ or a free boundary minimal surface in $B_{\mathrm{Eucl}}^{3}(0,\bar K)\cap \R^{3}_+$, with respect to $g$, such that $\partial\Sigma\setminus \partial \R^{3}_+ \subset \partial B_{\mathrm{Eucl}}^{3}(0,\bar K)$. Assume that 
\begin{enumerate}[label=\normalfont(\roman*)]
\item\label{bnct:ii} $s_\Sigma(0) \le 1/\bar K$;
\item\label{bnct:iii} there is $y'\in B(0,7\lambda)$ with $s_\Sigma(y') = 1$.
\end{enumerate}
Then, we have that
\begin{enumerate} [label=\normalfont(\arabic*)]
\item \label{c1t:pair-of-balls} either there are two stability balls $b,\, b'\subset B(0,\bar K)$ such that $3\lambda b\cap 3\lambda b'=\emptyset$ and $\radius(b),\radius(b') \in [2^{-(R+1)},2^{-R})$ for some $R\in[1000,\bar R]$;
\item \label{c2t:conic-annulus} or $\Sigma\cap A(0,1/(2K), 7\lambda K)$ is $\delta$-close to a cone in $\mathcal{G}=\mathcal{G}(2)$.
\end{enumerate}
\end{theorem3}
\end{case3}
\end{theorem}

\begin{proof}
If $\Sigma$ is a minimal hypersurface in $B^{n+1}_{\mathrm{Eucl}}(0,\bar K)$, then the theorem coincides with \cite{Song2019}*{Theorem~8}. Hence, let us assume to be in the free boundary setting, namely to work in $B^{n+1}_{\mathrm{Eucl}}(0,\bar K)\cap \R^{n+1}_+$. The proof is very similar to the one in the closed case, we report it here for completeness. Suppose by contradiction that, for every $k$ sufficiently large, there exists a free boundary minimal hypersurface $\ms_k^n \subset (B_{\mathrm{Eucl}}^{n+1}(0,\bar K)\cap\R^{n+1}_+, g_k)$, where $g_k$ is a metric $(1/k)$-close to the Euclidean metric, such that \ref{bnc:ii}, \ref{bnc:iii} and \ref{bnc:i} hold, $\Theta^{\Sigma_k}_{g_k}(0,10\lambda K)- \Theta^{\Sigma_k}_{g_k}(0,1/(3K)) \leq 1/k$ and such that \ref{c1:pair-of-balls} and \ref{c2:conic-annulus} do not hold.

The fact that \ref{c1:pair-of-balls} does not hold implies that the stability radius of $\Sigma_k$ is uniformly bounded below by $2^{-\bar{R}}$ in the annulus $B_{g_k}(0,10\lambda K)\setminus B_{g_k}(0,1/(3K))$, provided we choose $\bar{R}=\bar R(\lambda, K) >1000$ sufficiently large (see \cite{Song2019}*{equation~(18)}). This uses \ref{bnc:iii} and can be proven in the exact same way as \cite{Song2019}*{Fact in the proof Theorem~8} except that we work in $\R^{n+1}_+$ instead of $\R^{n+1}$.

Thus, in $A_{g_k}(0,1/(3K),10\lambda K)$, $\ms_k$ have uniformly bounded second fundamental form by \cref{thm:curvature-estimates} and uniformly bounded area by \ref{bnc:i}. Hence, by the compactness theorem \cite{GuangLiZhou2020}*{Theorem~6.1}, up to subsequence $\Sigma_k$ converges to a smooth (possibly nonproperly embedded) free boundary minimal hypersurface $\ms_\infty\subseteq A_{\mathrm{Eucl}}(0,1/(3K),10\lambda K)\cap\R^{n+1}_+$, with respect to the Euclidean metric $g_{\mathrm{Eucl}}$. 
In particular, we have that \[
\Theta_{{\mathrm{Eucl}}}^{\Sigma_\infty}(0,10\lambda K)-\Theta_{{\mathrm{Eucl}}}^{\Sigma_\infty}(0,1/(3K))=0.\] Thus, by the rigidity part of the monotonicity formula \cref{thm:monotonicity}, the hypersurface $\ms_\infty$ coincides in $A_{\mathrm{Eucl}}(0,1/(3K),10\lambda K)\cap\R^{n+1}_+$ with a (possibly nonproperly embedded) free boundary minimal cone $\Gamma$ tipped at $0$. Moreover, since $\Sigma_k$ satisfies \ref{bnc:i}, we get that then $n$-volume of $\Sigma_k$ is at most $\bar \CA \bar K^n/2$, which implies that $\Theta^{\Gamma}_{{\mathrm{Eucl}}}(0,1)\le \bar \CA$.

In fact, we have that $\Gamma\in \mathcal{G}$ since we are assuming that \ref{c1:pair-of-balls} does not hold and, if $\Gamma\notin \mathcal{G}$, then thanks to \cref{lem:TwoBallsBadCones} we can find two balls as in \ref{c1:pair-of-balls} with respect to the metric $g_k$ for $k$ sufficiently large.
However, by the convergence, $\ms_k$ is $\delta$-close to the cone $\Gamma$ in $A_{g_k}(0,1/(2K),7\lambda K)$ for $k$ sufficiently large, contradicting \ref{c2:conic-annulus}.

As a result, we obtain the existence of $\Omega$, $\varepsilon$, $\bar R$ such that the theorem holds. Observe that, from the proof, it follows that these constants do not depend on the variable $\bar r$, as far as we assume (as in the statement) that $\bar r \ge 1$.
\end{proof}

\begin{proof} [Proof of \cref{thm:BoundaryNonConcentration} for $n+1=3$]
If $\Sigma$ is a minimal surface in $B^3_{\mathrm{Eucl}}(0,\bar K)$, then the theorem coincides with \cite{Song2019}*{Theorem~14}. Hence, let us assume to be in the free boundary setting. Assuming by contradiction that the statement does not hold, we get a sequence of free boundary minimal surfaces $\Sigma_k$ in $B^3_{\mathrm{Eucl}}(0, 2^k)\cap \R^3_+$, with respect to a metric $g_k$ converging to the Euclidean metric on compact sets, such that $s_{\ms_k}(0)\le 2^{-k}$, \ref{bnct:iii} holds, and \ref{c1t:pair-of-balls} and \ref{c2t:conic-annulus} do not hold.
Following exactly the same argument as in the closed case, we get that for all $k$ large enough and $l>0$ it holds
\[
\sup_{y\in A_{g_k}(0,2^{-l},2^l)} \abs{A^{\ms_k}(y)} \le C\left( \dist_{g_k}(0,y) + \dist_{g_k}(0,y)^{-1} \right),
\]
for some constant $C>0$ independent of $k,l$.
This implies that $\ms_k$ converges to a (possibly nonproperly embedded) free boundary minimal lamination $\mathcal{L}$ in $\R^3_+\setminus\{0\}$ (see \cite{CarlottoFranz2020}*{Definition~3.1} for the definition) by \cite{CarlottoFranz2020}*{Theorem~3.5} (see also \cite{GuangLiZhou2020}*{Theorem~5.5}).
Let $L$ be a leaf whose closure contains $0$ (which exists because $s_{\ms_k}(0)\le 2^{-k}$).
Then, by \cite{CarlottoFranz2020}*{Lemma~3.7}, $L$ either has stable universal cover, or the convergence of $\Sigma_k$ to $L$ is locally smooth with multiplicity one. Observe that the second possibility cannot occur, because $s_{\ms_k}(0)\le 2^{-k}$ would then imply that $s_L(0) = 0$, contradicting \cref{lem:stab-continuous-positive}.
Thus, $L$ has stable universal cover, which implies that its closure is a plane or a half-plane (for example by \cite{CarlottoFranz2020}*{Corollary~3.8}).
As a result, all the leaves of $\mathcal{L}$ are flat. To prove this, one can reason as in the closed case at the level of the double of the lamination $\mathcal{L}$ (see e.g. \cite{GuangLiZhou2020}*{Lemma~5.4} or proof of \cite{CarlottoFranz2020}*{Corollary~3.8}). This concludes the proof, since it contradicts that \ref{c2t:conic-annulus} does not hold.
\end{proof}

\subsection{Construction of the almost conical region}
Let us assume to be in the setting described in \cref{sec:choice-param}. Moreover, assume that $\BO=\BO(n,\bar\CA)> 0$ is the constant given by \cref{lem:TwoBallsBadCones} and that $\mathcal{G} = \mathcal{G}(n,\bar \CA,\BO) = \mathcal{G}(n,\bar \CA) = \mathcal{G}(M,g,\CA)$ is the set of cones defined in \cref{sec:choice-cones}.
Then, we can generalize the definitions of pointed $\delta$-conical annuli and of telescopes given by \cite{Song2019}*{Sections 2.2.3 and 2.2.4} to the free boundary setting.
\begin{definition}
Given $\delta>0$, we say that $A(p,s,t)$ is a \emph{pointed $\delta$-conical annulus} if:
\begin{itemize}
\item either $B(p,t)\cap\partial M=\emptyset$ or $p\in\partial M$;
\item $s_\ms(p)<s\le t/2\le \bar r/2$;
\item $\Sigma\cap A(p,s,t)$ is $\delta$-close to a cone in $\mathcal{G}$.
\end{itemize}

Moreover, given another parameter $K>1000+\lambda$, we define
\[
\mathcal{A}_{\delta,K}\eqdef \left\{A(p,s,2s)\st  A\left(p,\frac{s}{K}, Ks\right) \text{ is a pointed $\delta$-conical annulus}\right\}
\]
and
\[
\mathcal{A}_{\delta,K}^{\mathrm{bis}}\eqdef \left\{A(p,s,2s)\st A\left(p,\frac{100s}{K}, \frac{Ks}{100}\right) \text{ is a pointed $\delta$-conical annulus}\right\}.
\]

\end{definition}

\begin{definition}
We say that $T\subset M$ is a \emph{$(\delta,K)$-telescope} if there exist annuli $A(p_i,s_i,2s_i)\in\mathcal{A}_{\delta,K}^{\mathrm{bis}}$ for $i=1,\ldots, I$ such that
\[
T = \bigcup_{i=1}^I A(p_i,s_i,2s_i)
\]
and
$\overline{B}(p_i,s_i)\subset B(p_{i+1},s_{i+1})$, $\overline{B}(p_{i+1},s_{i+1})\subset B(p_{i},2s_{i})$, 
and $\overline{B}(p_i,2s_i)\subset B(p_{i+1},2s_{i+1})$ for all $i=1,\ldots,I-1$.
\end{definition}

The following proposition says that, by suitably choosing $\delta$ and $K$, we can cover the union of the annuli in $\mathcal{A}_{\delta,K}$ with a finite union of $(\delta,K)$-telescopes, for which we can well control the topology. The analogous result in the closed case can be found in \cite{Song2019}*{Lemma~7}.

\begin{proposition} \label{prop:ConicalRegion}

In the setting described in \cref{sec:choice-param}, there exist $\delta=\delta(M,g,\CA)>0$ and $K=K(M,g,\CA)>0$ such that the following statement holds.

There exists a finite family $\mathcal{T}$ of $(\delta,K)$-telescopes at positive distance from one another such that
\[
\bigcup_{An\in\mathcal{A}_{\delta,K}} An \subset \bigcup_{T\in\mathcal{T}} T.
\]
Moreover, we have that
\begin{enumerate} [label={\normalfont(\arabic*)}]
\item\label{cr:topology-telescopes} for any $T\in\mathcal{T}$, each connected component of $\Sigma\cap T$ are diffeomorphic to a cone in $\mathcal{G}$, and the sum of the Betti numbers of any connected component of $\Sigma\cap T$ and $\partial\Sigma \cap T$ is bounded above by a constant $\alpha=\alpha(M,g, \CA)$;

\item\label{cr:trivial-intersection} for any $T\in\mathcal{T}$ and any {stability balls $B_1,\ldots,B_J\in\mathcal{B}_\ms$, the connected components of 
\[
\Sigma\cap T\cap \bigcap_{j=1}^J B_j \quad \text{ and } \quad 
\partial\Sigma\cap T\cap \bigcap_{j=1}^J B_j
\]}
are topological $n$-discs and $(n-1)$-discs, respectively;
\item\label{cr:controlled-overlap} {for any $B\in\mathcal{B}_\ms$, there are at most two $(\delta,K)$-telescopes $T\in \mathcal{T}$ such that $B\cap T\not=\emptyset$;}
\item \label{cr:sheeting-telescopes} there exists $C_{CT}=C_{CT}(M,g,\CA)>0$ such that
\[
\abs{\{\text{components of }\Sigma\cap T\}}, \abs{\{\text{components of }\partial \Sigma\cap T\}}  \le C_{CT} \CA.
\]
\end{enumerate}
\end{proposition}

\begin{case3}
In the case $n+1=3$, exactly the same statement holds with the constants $\delta=\delta(M,g)>0$, $K=K(M,g)>0$, $\alpha =1$, $\bar \beta=\bar \beta(M,g)$, $C_{CT}=C_{CT}(M,g)>0$ not depending on $\CA$. The proof is exactly the same as in the higher dimensional case, using that $\mathcal{G}=\mathcal{G}(M,g)$ and $C_{CE}=C_{CE}(M,g)$ do not depend on $\CA$.
\end{case3}

The proof of the proposition relies on the following lemmas.

\begin{lemma} \label{lem:ConicalStabBound}
There exist $K_1 = K_1(M,g,\CA)>0$ and $\delta=\delta(M,g,\CA) > 0$ small enough such that, if $A(p,s,t)$ is a pointed $\delta$-conical annulus, then
\[
K_1t\le s_\ms(q)\le \left( 1-\frac 1\lambda \right)^{-1}\left(\frac t\lambda + s\right)
\]
for all $q\in M$ such that $B(q,s_\ms(q))\cap \partial B(p,t)\not=\emptyset$.
\end{lemma}
\begin{proof}
The proof is exactly the same as the one of equation (9) in \cite{Song2019}*{Section~2.2.3}, given \cref{lem:BoundBelow} (in place of \cite{Song2019}*{(8)}).
\end{proof}

\begin{lemma} [cf. \cite{Song2019}*{Lemma 5}]
\label{lem:ScaleOfAnnuli}
For any $K_2>0$, we can choose $K=K(M,g,\CA, K_2)>0$ large enough such that the following property holds. Let $A(p,u,2u), A(q,v,2v)\in \mathcal{A}_{\delta,K}^{\mathrm{bis}}$ such that
\[
\operatorname{dist}(A(p,u,2u), A(q,v,2v)) \le \frac u3.
\]
Then, either $A(p,u,2u)$ and $A(q,v,2v)$ do not intersect the boundary $\partial M$, or $p,q\in\partial M$. Moreover,
\[
\dist(p,q) \le \frac{u}{K_2}
\quad \text{and} \quad \frac u4 < v < 4u.
\]
\end{lemma}
\begin{proof}
Assume by contradiction that $v\le u/4$, then by assumption $B(q,v)\subset A(p,u/6,3u)$. As a result, since $A(p,100u/K, Ku/100)$ is a pointed $\delta$-conical annulus, by \cref{lem:ConicalStabBound} we get that $s_\ms(q) \ge K_1u/6$. However, this contradicts that $s_\ms(q) < 100v/K$ (which follows by the fact that $A(q,100v/K, Kv/100)$ is a pointed $\delta$-conical annulus) for $K$ sufficiently large depending on $K_1$. Therefore we get that $v>u/4$. Analogously, one can prove that $v < 4u$. Moreover, given a constant $K_2>1$, with a similar argument, one can prove that $K$ can be also chosen sufficiently large, depending also on $K_2$, such that  $\dist(p,q)\le u/K_2$.

Now, assume that $p\in\partial M$. Then
\[\dist(q,\partial M) \le \dist(p,q) \le \frac{u}{K_2}\le \frac{4v}{K_2} < \frac{Kv}{100}.\] 
Therefore, $A(q,100v/K, Kv/100)$ intersects $\partial M$ and thus it is a pointed $\delta$-conical (half-)annulus and $q\in \partial M$, since $\delta$-conical annuli do not intersect the boundary. Analogously once can prove that, if $q\in \partial M$, then $p\in\partial M$.
\end{proof}

\begin{proof} [Proof of \cref{prop:ConicalRegion}]
    The proof is the same as in \cite{Song2019}*{Lemma~7} using \cref{lem:ScaleOfAnnuli} instead of \cite{Song2019}*{Lemma~5} since this prevents telescopes of boundary type to interact with interior ones. The additional properties \ref{cr:trivial-intersection} and \ref{cr:controlled-overlap} can be proven in a similar way to how \cite{Song2019}*{Lemma~7~(ii) and (iv)} are proven, given \cref{prop:write-as-graph} in place of \cite{Song2019}*{(6)}, \cref{lem:ConicalStabBound} in place of \cite{Song2019}*{(9)} and \cref{lem:ScaleOfAnnuli} in place of \cite{Song2019}*{Lemma~5}.
    Property~\ref{cr:sheeting-telescopes} is proven similarly to \cref{prop:sheeting-results}, see also \cite{Song2019}*{paragraph after (20)}.
    To prove \ref{cr:topology-telescopes} one can argue as in \cite{Song2019}*{Lemma~7~(i)} to show that the diffeomorphism type of a connected component of $T\cap\ms$ corresponds to a diffeomorphism type of a cone in $\mathcal{G}$. Since those cones have bounded area and curvature, the compactness theorem \cite{GuangLiZhou2020}*{Theorem~6.1} implies there is only finitely many such diffeomorphism types.
\end{proof}

Thanks to the previous result, we now have all the tools to define the almost conical region.

\begin{definition} \label{def:conical}
In the setting of \cref{sec:choice-param}, and let $\delta=\delta(M,g,\CA)>0$ and $K=K(M,g,\CA)>0$ given by \cref{prop:ConicalRegion}.
Then, we define the \emph{almost conical region} $\conical$ associated to $\Sigma$ as
\[
\conical \eqdef \bigcup_{T\in\mathcal{T}} T,
\]
where $\mathcal{T}$ is the family of telescopes given by \cref{prop:ConicalRegion}.
We call the complement $M\setminus\conical$ the \emph{concentration region} associated to $\Sigma$.

\begin{case3}
Again, when $n+1=3$, $\delta=\delta(M,g)$ and $K=K(M,g)$ do not depend on the area bound $\CA$.
\end{case3}
\end{definition}

\begin{remark}
In \cite{Song2019}, when $n+1=3$, the almost conical region is called \emph{almost flat region} since the minimal two-dimensional cones in $\R^3$ are flat. Here, we keep the same notation in the case $n+1=3$ since we treat this case alongside the higher-dimensional setting.
\end{remark}

\begin{remark}
Observe that the almost conical region is not uniquely defined, since it depends on the telescopes chosen in \cref{prop:ConicalRegion} (see also the comment after \cite{Song2019}*{Definition 2.2}).
\end{remark}

\section{The concentration region} \label{sec:concentration-region}

Assume to be in the setting of \cref{sec:choice-param}, and assume that the almost conical region $\conical$ is defined as in \cref{def:conical}. The goal of this section is to prove that the topology of the concentration region $M\setminus\conical$ is controlled by the index of the free boundary minimal hypersurface $\Sigma$. More precisely, we will show in \cref{thm:DisjointIndex} that we can bound from above the cardinality of a disjoint family of stability balls centered in the concentration region by the index of $\Sigma$.

\subsection{Dichotomy property in the concentration region}

The idea why we can control the size of a disjoint family of stability balls centered in the concentration region is that there cannot be a chain of stability balls intersecting each other and becoming smaller and smaller, or, if there is, we can ``substitute'' this chain with two stability balls that are sufficiently disjoint.
The following lemma is a precise statement of what we roughly described in words.

\begin{lemma} \label{lem:caseA3}
Let us assume to be in the setting described in \cref{sec:choice-param}. Moreover, let us assume that the almost conical region $\conical$ is defined as in \cref{def:conical}. Then, there exist $m=m(M,g,\CA)\in\N$ and $\bar R=\bar R(M,g,\CA)>0$ sufficiently large such that the following statement holds.
Let $B,b_1,\ldots,b_m$ be stability balls such that 
\begin{enumerate} [label={\normalfont(\roman*)}]
\item $b_i$ is centered in $M\setminus\conical$ for all $i=1,\ldots, m$;
\item \label{ca3:decay-radius} $\radius(b_i) <\radius(b_{i-1})/2$ for all $i=2,\ldots, m$ and $\radius(b_1)<\radius(B)/2$;
\item $6\lambda b_i\cap 3\lambda b_m\not=\emptyset$ for all $i=1,\ldots,m-1$;
\item $3\lambda b_i\cap 1.1\lambda B\not=\emptyset$ for all $i=1,\ldots,m$.
\end{enumerate}
Then, there are two stability balls $b_B,b_B'$ such that
\begin{enumerate}[label={\normalfont(\arabic*)}]
\item $\radius(b_B)=\radius(b_B') \ge 2^{-\bar R}\radius(b_m)$;
\item $3\lambda b_B\cap 3\lambda b_B'=\emptyset$;
\item \label{ca3:InB} $3\lambda b_B\cup 3\lambda b_B'\subset 2\lambda B$.
\end{enumerate}

\begin{case3}
When $n+1=3$, the same statement holds with $m=m(M,g)$ and $\bar R=\bar R(M,g)$ that do not depend on $\CA$.
\end{case3}
\end{lemma}
\begin{proof}
Set $\bar K$ as in the proof of \cref{thm:BoundaryNonConcentration}.
We first note that it is sufficient to prove the lemma with the assumption~\ref{ca3:decay-radius} replaced by
\begin{enumerate} [label={\normalfont(\roman*')}, start=2]
    \item\label{ca3:decay-radius-new} $\radius(b_i) <\radius(b_{i-1})/\bar{K}$ for all $i=2,\ldots, m$ and $\radius(b_1)<\radius(B)/\bar{K}$.
\end{enumerate}
Indeed, if we prove the existence of $m,\bar R$ satisfying the statement with \ref{ca3:decay-radius} replaced by \ref{ca3:decay-radius-new}, then we get the original statement by choosing constants $m\log_2(\bar K),\bar R$.
Moreover, by discarding some of the balls, thus choosing $m$ possibly larger depending on $M,g,\CA$, we can assume that at scale $\radius(b_1)$ the metric $g$ is $\varepsilon$-close to the Euclidean metric in the $C^5$-topology, where $\varepsilon$ is given by \cref{thm:BoundaryNonConcentration}.

Now, let $x$ be the center of $b_m$ and let $m'=\lfloor (m-1)/2 \rfloor$. 
Moreover, let us denote by $r_i \eqdef \radius(b_i)$ the radius of $b_i$ for all $i=1,\ldots,m$. 
We distinguish two cases:
\begin{enumerate} [label={\normalfont(\Alph*)}]
\item\label{c1:interior} $d(x,\partial M) > \bar K r_{m'}$;
\item\label{c2:boundary} $d(x,\partial M) \le \bar K r_{m'}$.
\end{enumerate}
Observe that \ref{c1:interior} implies that $b(x,\bar K r_m)\subset b(x,\bar Kr_{m-1})\subset\ldots \subset b(x,\bar K r_{m'})$ do not intersect $\partial M$. Then, this case follows essentially as in \cite{Song2019}*{Proof of Proposition~10, Case~A.3} and it is simpler than \ref{c2:boundary}, thus we leave the details of this case to the reader.

Hence, let us assume that \ref{c2:boundary} holds. Let $\bar x\in\partial M$ be the projection of $x$ onto $\partial M$, namely the point in $\partial M$ closest to $x$.
We first show that, by choosing $m'=m'(M,g,\CA)>0$ sufficiently large, we can assume that there exists $i=1,\ldots,m'-2$ such that
\begin{equation} \label{eq:density-drop}
\Theta(\bar x, 10\lambda K r_i) - \Theta\left(\bar x, \frac{r_i}{3K} \right) < \BU,
\end{equation}
where $\BU>0$ is given by \cref{thm:BoundaryNonConcentration}.
If $\Theta(\bar x, 10\lambda K r_i) - \Theta\left(\bar x, {r_i}/{(3K)} \right)<0$ for some $i=1,\ldots,m'-2$, then the previous inequality holds trivially for such $i$. Otherwise, if $\Theta(\bar x, 10\lambda K r_i) - \Theta\left(\bar x, {r_i}/{(3K)} \right)>0$ for all $i=1,\ldots,m'-2$, then, we have that
\begin{align*}
\sum_{i=1}^{m'-2} \Theta(\bar x, 10\lambda K r_i) - \Theta\left(\bar x, \frac{r_i}{3K} \right) &\le\sum_{i=1}^{m'-2} e^{C_{MF}r_i/(3\lambda)}\left(\Theta(\bar x, 10\lambda K r_i) - \Theta\left(\bar x, \frac{r_i}{3K} \right) \right)\\
&\le  \sum_{i=1}^{m'-2} e^{C_{MF}10\lambda r_i}\Theta(\bar x, 10\lambda K r_i) - e^{C_{MF}r_i/(3\lambda)}\Theta\left(\bar x, \frac{r_i}{3K} \right)\\
&\le e^{C_{MF}10\lambda r_1}\Theta(\bar x, 10\lambda K r_1) - e^{C_{MF}r_{m'-2}/(3\lambda)}\Theta\left(\bar x, \frac{r_{m'-2}}{3K} \right)\\
&\le e^{C_{MF} 10\lambda \bar r} C_{DE} \Haus^n(\Sigma),
\end{align*}
where we used \cref{thm:monotonicity} and \cref{cor:density-estimate} in the last two inequalities, respectively.
Therefore, there exists $i\in\{1,\ldots,m'-2\}$ such that
\begin{equation}\label{eq:DensityDrop}
\Theta(\bar x, 10\lambda K r_i) - \Theta\left(\bar x, \frac{r_i}{3K} \right) \le \frac{e^{C_{MF} 10\lambda \bar r} C_{DE} \Haus^n(\Sigma)}{m'-2},
\end{equation}
which is smaller than $\BU$ by choosing $m'$ large enough depending on $M,g,\CA$.

By \cref{lem:stab-continuous-positive}, observe that
\begin{equation} \label{eq:stab-small}
\begin{split}
s_\Sigma(\bar x) &\le s_\Sigma (x) + \frac{1}{\lambda} d(x,\bar x) \le r_m + \frac{\bar K}\lambda r_{m'} < \left(\frac 1{\bar K^2} + \frac 1{\lambda\bar K}\right) r_{m'-2}\le \frac{r_{m'-2}}{\bar K}  \le \frac{r_i}{\bar K}.
\end{split}
\end{equation}
Moreover, denoting by $y$ the center of $b_i$, it holds that
\begin{equation} \label{eq:dist-small}
\begin{split}
d(\bar x, y) &\le d(\bar x, x) + d(x, y) \le \bar K r_{m'} + (6\lambda r_i + 3\lambda r_m)\le \left(6\lambda + \frac 1{\bar K} + \frac{3\lambda}{\bar K^2}\right) r_i \le 7\lambda r_i.
\end{split}
\end{equation}
Finally, note that $\bar K r_i \le \radius(B) \le \bar r$. Therefore, by \cref{cor:density-estimate}, we have that 
\begin{equation} \label{eq:area-bound}
\Haus^n(\Sigma\cap B(\bar x, \bar K r_i) )\le \omega_n C_{DE} \CA \bar K^n r_i^n=\frac{\bar \CA}{2} \bar K^nr_i^n.
\end{equation}

Recall that we chose $m$ sufficiently large such that at scale $r_1$, and so also at scale $r_i\le r_1$, the metric $g$ is $\varepsilon$-close to the Euclidean metric in the $C^5$-topology, where $\varepsilon$ is given by \cref{thm:BoundaryNonConcentration}. Namely, the $r_i^{-1}$ dilation of the metric $g$ on $B(\bar x, \bar Kr_i)$ is $\varepsilon$-close to the Euclidean metric in the $C^5$-topology.
Moreover, note that by \eqref{eq:density-drop}, \eqref{eq:stab-small}, \eqref{eq:dist-small}, and \eqref{eq:area-bound}, all the assumptions of \cref{thm:BoundaryNonConcentration} are satisfied in the $r_i^{-1}$-dilation of $B(\bar x,\bar K r_i)$ by identifying $\bar x$ with $0$, and $y$ with $y'$. Therefore, we get that 
\begin{enumerate} [label=\normalfont(\roman*)]
\item\label{cat:twoballs} either there are two stability balls $b_B,\, b_{B'}$ with $b_B,b_B'\subset B(\bar x,\bar K r_i)$, $3\lambda b_B\cap 3\lambda b_{B'}=\emptyset$ and $\radius(b_B)=\radius(b_B') \in [2^{-(R+1)}r_i,2^{-R}r_i)$ for some $R\in[1000,\bar R]$, where the constant $\bar R=\bar R(M,g,\CA,\lambda)>0$ is given by the theorem;
\item\label{cat:conical} or $\Sigma\cap A(\bar x,r_i/(2K), 7\lambda Kr_i)$ is $\delta$-close to a cone in $\mathcal{G}$.
\end{enumerate}
If case \ref{cat:conical} occurs, then $A(\bar x, r_i/2, 7\lambda r_i)\in\mathcal{A}_{\delta,K}$ by definition and therefore $A(\bar x, r_i/2, 7\lambda r_i)\subset\conical$ by \cref{prop:ConicalRegion}. However, this contradicts the fact that $y\not\in\conical$, since $y\in A(\bar x, r_i/2, 7\lambda r_i)$ thanks to \eqref{eq:dist-small} and
\[
d(\bar x, y) \ge \lambda (s_\Sigma(y) - s_\Sigma(\bar x)) \ge \lambda r_i\left(1 - \frac{1}{\bar K}\right)  \ge \frac{r_i}{2}.
\]

As a result, \ref{cat:twoballs} holds and we obtain two stability balls $b_B,b_{B'}$ such that $3\lambda b_B\cap 3\lambda b_{B}'=\emptyset$, $3\lambda b_B\cup 3\lambda b_{B}'\subset B(\bar x,\bar Kr_i)$, and \[\radius(b_B)=\radius(b_B') \ge 2^{-\bar R-1}r_i \ge 2^{-\bar R} r_m = 2^{-\bar R}\radius(b_m).\] To conclude the proof, we then just need to check \ref{ca3:InB}.
By using that $3\lambda b_m\cap1.1\lambda B \not=\emptyset$ we get that, if $z\in3\lambda b_B\cup 3\lambda b_B'\subset B(\bar x,\bar Kr_i)$ and $w$ is the center of $B$, then
    \begin{align*}
        d(z,w) 
        &\leq  \bar Kr_i +d(\bar x, x) + d(x,w)
        \le \bar K r_i + \bar K r_{m'} + (3\lambda r_m + 1.1\lambda \radius(B))\\
        &\le (2 + 3\lambda / \bar K + 1.1\lambda)\radius(B) \le 2\lambda \radius(B),
    \end{align*}
    which concludes the proof.
\end{proof}

\begin{proof} [Proof of \cref{lem:caseA3} for $n+1=3$]
The proof is simpler than the higher dimensional case. Indeed, it works the same but we do not have to ensure the density estimate \eqref{eq:DensityDrop} and the area bound \eqref{eq:area-bound} to apply \cref{thm:BoundaryNonConcentration} in dimension $n+1=3$. As a result, one can check that the proof goes through without the dependence on $\CA$.
\end{proof}

Thanks to the previous lemma, we are now able to show that any suitable (i.e., layered $3\lambda$-disjoint) cover of the concentration region by stability balls satisfies a dichotomy property.
Let us first introduce this property, called property \PJ{}, and then prove the result, which is the analogue of \cite{Song2019}*{Proposition~10} in the free boundary setting.

\begin{definition}
Let $J>0$ be a positive integer, let $\mathcal{B}$ be a layered $3\lambda$-disjoint family of stability balls and let $B\in\mathcal{B}_k$ for some $k\ge 0$.
We say that $(\mathcal{B}, B)$ satisfies property \PJ{} if one of the following two cases occurs:
\begin{enumerate} [label={\normalfont(\arabic*)}]
\item \label{pj:intersection} either the size of the set
\[
\{\hat b\in \mathcal{B}_{k+u}\st u\ge 0,\ 3\lambda \hat b\cap \lambda B \not=\emptyset\}
\]
is bounded above by $J$;
\item \label{pj:twoballs} or there are $v=v(B) > 1000$ and two stability balls $b_B,b_B'\in\mathcal{F}_{k+v}$ centered in $M$ such that $3\lambda b_B\cap 3\lambda b_B'=\emptyset$, $3\lambda b_B\cup3\lambda b_B'\subset 2\lambda B$ and such that the size of 
\[
\{\hat b\in \mathcal{B}_{k+u} \st 0\le u\le v,\ 3\lambda \hat b\cap(3\lambda b_B\cup 3\lambda b_B')\not=\emptyset) \}
\]
is bounded above by $J$.
\end{enumerate}
\end{definition}

{
\begin{definition}[\cite{Song2019}*{Section~2.5.1}]
Let $\mathcal{B}$ be a layered $3\lambda$-disjoint family of stability balls. We say that $\mathbf{B}\subset \mathcal{B}$ is a basis of $\mathcal{B}$ if, for all $k\ge 0$, $\mathbf{B}\cap\mathcal{B}_k$ is a maximal subset of $\mathcal{B}_k$ such that 
\[
3\lambda b\cap 3\lambda b'=\emptyset
\]
for all $b\in\mathbf{B}\cap\mathcal{B}_k$ and $b'\in\bigcup_{j < k}\mathcal{B}_j$.
\end{definition}
}

\begin{proposition}
\label{prop:Property-PJ}
Let us assume to be in the setting described in \cref{sec:choice-param}. Moreover, let us assume that the almost conical region $\conical$ is defined as in \cref{def:conical}. Then, there exists $J=J(M,g,\CA)\in\N$ such that the following statement holds.

Let $\mathcal{B}{\subset\mathcal{B}_\ms}$ be a layered $3\lambda$-disjoint family of stability balls and let $\mathbf{B}$ be a basis of $\mathcal{B}$. Moreover, assume that every ball in $\mathcal{B}\setminus\mathbf{B}$ is centered in $M\setminus\conical$.
Then $(\mathcal{B}, B)$ satisfies property \PJ{} for all $B\in\mathbf{B}$.
\end{proposition}

\begin{case3}
When $n+1=3$, the same result holds with $J=J(M,g)$ not depending on $\CA$. This follows from the fact that the constants in \cref{lem:caseA3} (which is the core of the proof below) do not depend on $\CA$.
\end{case3}

\begin{proof}
The proof is similar to the proof of \cite{Song2019}*{Proposition~10}, we report it (with some changes and rearrangements) for completeness.

Choose $J=C_{NL} + 2(m + \bar R)C_{PL}$, where $C_{PL}$ and $C_{NL}=C_{NL}(1000)$ are given by \cref{lem:fact1,lem:fact2}, $m$ and $\bar R$ are given by \cref{lem:caseA3}. 
First, note that we can assume that there exists $\bar b \in\mathcal{B}_{k+v}$, for some $v>1000$, such that $3\lambda \bar b\cap \lambda B\not=\emptyset$, since otherwise case \ref{pj:intersection} holds, using \cref{lem:fact2} since $J>C_{NL}$.
Therefore, thanks to \cite{Song2019}*{Proof of Proposition~10, Case A.1}, observe that, if there exists $b\in\mathcal{B}_{k+v}$ for some $v>1000$ such that $3\lambda b\cap 1.1\lambda B =\emptyset$ and $3\lambda b\subset 2\lambda B$, then case \ref{pj:twoballs} holds.

As a result, we can assume that all balls $b\in\mathcal{B}_{k+v}$, for some $v>1000$, such that $3\lambda b\subset 2\lambda B$ satisfy $3\lambda b\cap 1.1\lambda B\not=\emptyset$. On the other hand, note that also the converse is true: every ball $b\in\mathcal{B}_{k+v}$ with $v>1000$ such that $3\lambda b\cap 1.1\lambda B \not=\emptyset$ satisfies $3\lambda b\subset 2\lambda B$, since $\radius(b)\le 2^{-1000}\radius(B)$.

Let $1000<v_1<v_2<\ldots<v_N$, for some $N> 0$, be the numbers such that for all $i=1,\ldots,N$ there exists at least one ball $b\in \mathcal{B}_{k+v_i}$ whose $3\lambda$-dilation intersects $ 1.1\lambda B$.
Thanks to what we said before, this is equivalent to the fact that $3\lambda b\subset 2\lambda B$.

Consider $m$ given by \cref{lem:caseA3}.
    Assume that there are two distinct balls $b_1,b_2\in \mathcal{B}_{k+{v_i}}$ with $3\lambda b_1\cap1.1\lambda B\neq \emptyset$ and 3$\lambda b_1\cap1.1\lambda B\neq\emptyset$, for some $i\in\{1,\ldots,m\}$. Then $3\lambda b_1\cap3\lambda b_2=\emptyset$ (because the family $\mathcal{B}$ is layered $3\lambda$-disjoint), $3\lambda b_1\cup3\lambda b_2 \subset 2\lambda B$ and the size of 
    \[
    \{\hat b\in \mathcal{B}_{k+u}\st 0\le u\le v_i,\ 3\lambda\hat b\cap(3\lambda b_1\cup3\lambda b_2) \not=\emptyset\}
    \]
    is bounded above by $C_{NL} + 2mC_{PL} < J$. Therefore, case \ref{pj:twoballs} holds.
    
    Thus, we may assume that for each number $v_i$ there is exactly one ball $b_i\in \mathcal{B}_{k+v_i}$ satisfying $3\lambda b_i\cap1.1\lambda B\neq\emptyset$, or equivalently $3\lambda b_i\subset 2\lambda B$.
    If $N<m$ then case \ref{pj:intersection} holds, using $C_{NL} + m < J$. So let us assume $N\geq m$ too.

If $6\lambda b_i\cap 3\lambda b_m =\emptyset$ for some $i=1,\ldots,m-1$, then case \ref{pj:twoballs} holds as in \cite{Song2019}*{Proof of Proposition~10, Case A.2}. Therefore, we can assume that $6\lambda b_i\cap 3\lambda b_m\not=\emptyset$ for all $i=1,\ldots,m-1$.
As a result, we can apply \cref{lem:caseA3} and obtain stability balls $b_B,b_B'$ such that $3\lambda b_B\cap 3\lambda b_B'=\emptyset$, $3\lambda b_B\cup 3\lambda b_B'\subset 2\lambda B$, and $\radius(b_B) = \radius(b_B') \ge 2^{-\bar R}\radius(b_M)$. Therefore, case \ref{pj:twoballs} of property \PJ{} holds, since the number of balls $\hat b\in\mathcal{B}$ with radius between $\radius(B)$ and $\radius(b_B) = \radius(b_B')$ such that $3\lambda\hat b\cap (3\lambda b_B\cup3\lambda b_B')\not=\emptyset$ is bounded above by $C_{NL} + m + 2\bar RC_{PL} < J$.
\end{proof}

\subsection{Index--topology estimate for the concentration region}

Thanks to the previous results, we are now able to prove the aforementioned estimate relating topology and index of a free boundary minimal hypersurface. The corresponding result in the closed case can be found in \cite{Song2019}*{Corollary~12}.

\begin{theorem} \label{thm:DisjointIndex}
Let us assume to be in the setting described in \cref{sec:choice-param}. Moreover, assume that the almost conical region $\conical$ associated to $\Sigma$ is defined as in \cref{def:conical}.
    Then, there exists $C_{MI}=C_{MI}(M,g,\CA)>0$ such that, for any disjoint family $\mathcal{B}{\subset\mathcal{B}_\ms}$ of stability balls centered in the concentration region $M\setminus\conical$, it holds 
    \begin{equation*}\label{Morse-index-bound}
        \abs{\mathcal{B}} \le C_{MI}(\ind(\ms)+1).
    \end{equation*}

\begin{case3}
In this case, the same result hold with $C_{MI}=C_{MI}(M,g)$ not depending on $\CA$. Indeed, one can see from the proof that $C_{MI}$ depends on the constants $C_{VB},C_{LF}, J$ from \cref{lem:Besicovitch-type-lemma,prop:we-can-layer,prop:Property-PJ} respectively, which are independent of $\CA$ for $n+1=3$.
\end{case3}
\end{theorem}
\begin{proof}
First, we classify the balls of $\mathcal{B}$ into two types: $ \mathcal{B} = \mathcal{R}\cup\mathcal{S}$, where $\mathcal{R}$ contains those balls with radius $\bar{r}$, and $\mathcal{S}$ those with radius strictly smaller than $\bar{r}$.
By \cref{lem:Besicovitch-type-lemma}\ref{btl:ii}, we have that
    \begin{equation}\label{eq:EstR}
   \abs{\mathcal{R}} \le C_{VB} = C_{VB}(M,g).
    \end{equation}

    As a result, we now just need to bound $\abs{\mathcal{S}}$. In particular we prove that there exists a family $\mathcal{S}'$ of stability balls with radius strictly less than $\bar r$ such that $\abs{\mathcal{S}'}\ge \abs{\mathcal{S}}/(2C_{LF}(J+1))$ and $\operatorname{dist}(\lambda b_1,\lambda b_2)>0$ for all distinct $b_1,b_2\in\mathcal{S}'$. Here $C_{LF} = C_{LF}(M,g)>1$ is given by \cref{prop:we-can-layer} and $J=J(M,g,\CA)$ is given by \cref{prop:Property-PJ}.

    This suffices to conclude the proof of the theorem. Indeed, since $\operatorname{dist}(\lambda b_1,\lambda b_2)>0$ for all $b_1,b_2\in\mathcal{S}'$ and $\lambda' b$ is unstable for all $b\in\mathcal{S}'$ and $\lambda'>\lambda$ (by definition of $s_\ms$), there exist $\abs{\mathcal{S}'}$ disjoint regions where $\ms$ is unstable. Hence, we have
    \[
        \ind(\ms)
            \geq \abs{\mathcal{S}'}
            \geq \frac{1}{2(J+1)C_{LF}}\abs{\mathcal{S}}.
    \]
    Together with \eqref{eq:EstR}, this implies that
    \[
    \abs{\mathcal{B}} = \abs{\mathcal{R}} + \abs{\mathcal{S}} \le C_{VB} + 2(J+1)C_{LF}\ind(\ms).
    \]
    So it suffices to take $C_{MI}= \max\{C_{VB},2(J+1)C_{LF}\}$.

    The proof of the existence of the family $\mathcal{S}'$ is exactly the same as the proof of the closed case \cite{Song2019}*{Proposition~11}, given \cref{prop:Property-PJ} in place of \cite{Song2019}*{Proposition 10}. Therefore, we report here just a sketch and we refer to Song's paper for the full proof.

    First, thanks to \cref{lem:fact1} and \cref{lem:fact2}, we can extract a layered $3\lambda$-disjoint subfamily $\mathcal{S}''\subset \mathcal{S}$ such that
    \[
        \abs{\mathcal{S}}\leq C_{LF}\abs{\mathcal{S}''},    
    \]
    where $C_{LF} = C_{LF}(M,g,\CA)>0$ is given by \cref{prop:we-can-layer}.
    Then, by \cref{prop:Property-PJ}, we know that there exists $J=J(M,g,\CA)>0$ such that $(\mathcal{S}'',B)$ satisfies the property \PJ{} for all $B\in\mathbf{B}$ where $\mathbf{B}$ is a basis for~$\mathcal{S}''$.

    Then, one can use the property \PJ{} as a compensating procedure for a greedy algorithm trying to find the family $\mathcal{S}'$. Namely, the idea is that we can classify the balls in $\mathbf{B}$ in two types depending on whether case \ref{pj:intersection} or case \ref{pj:twoballs} holds for $B\in\mathbf{B}$.
    In the first case, we remove the balls whose $3\lambda$-dilations intersect $\lambda B$ (which are at most $J$). If the second case occurs, then we change $B$ by the two smaller balls $b_B$ and $b_B'$ and we remove the at most $J$ balls in \ref{pj:twoballs}.
    Then we can iterate this again by finding a basis of the new family of balls we produced until this process eventually terminates.
\end{proof}

\section{Proof of the main theorem} \label{sec:proof}

The aim of this section is proving \cref{thm:main}. Hence, let us assume to be in the setting described in \cref{sec:choice-param}. Moreover, assume that the almost conical region $\conical$ is defined as in \cref{def:conical}.

Applying \cref{thm:Besicovitch} to the cover $\{B(p,s_\ms(p)){\in\mathcal{B}_\ms}\st p\in M\setminus\conical\}$ of $M\setminus\conical$, we obtain $C_{BC}$ subfamilies $\mathcal{B}^{(1)},\ldots,\mathcal{B}^{(C_{BC})}$ of disjoint balls that still cover $M\setminus \conical$.
Note that $\mathcal{B}^{(i)}$ is finite for all $i=1,\ldots,C_{BC}$. Indeed, using that $\mathcal{B}^{(i)}$ consists of disjoint balls, we have
\[
\Haus^{n+1}(M\setminus \conical) \ge \sum_{B\in\mathcal{B}^{(i)}} \Haus^{n+1}(B)  \ge \abs{\mathcal{B}^{(i)}} \inf_{p\in M} \Haus^{n+1}(B(p,s_\ms(p))),
\]
where $\inf_{p\in M} \Haus^{n+1}(B(p,s_\ms(p))) > 0$ by \cref{lem:stab-continuous-positive} and $\Haus^{n+1}(M\setminus\conical)\le \Haus^{n+1}(M)<+\infty$ by compactness. Without loss of generality we can assume that $\mathcal{B}^{(1)}$ is the largest of these families.

Now, observe that 
\[
\mathcal{U}_\ms\eqdef \bigcup_{i=1}^{C_{BC}} \bigcup_{B\in\mathcal{B}^{(i)}} \{\text{components of }\ms\cap B \} \cup \bigcup_{T\in\mathcal{T}} \{\text{components of }\ms\cap T\}
\]
is an $\alpha$-almost acyclic cover of $\Sigma$ with overlap at most $C_{BC}+1$, where $\alpha$ is given by \cref{prop:ConicalRegion}.
Indeed, for all $T\in\mathcal{T}$, the connected components of $\Sigma\cap T$ have sum of the Betti numbers bounded above by $\alpha$ by \cref{prop:ConicalRegion}\ref{cr:topology-telescopes}, and any other finite intersection of elements of $\mathcal{U}_\Sigma$ has trivial topology by \cref{cor:acyclic-cover} and \cref{prop:ConicalRegion}\ref{cr:trivial-intersection}.
Moreover, $\mathcal{U}_\ms$ has overlap at most $C_{BC}+1$ because it is the union of the $C_{BC}+1$ families consisting of disjoint subsets of $\Sigma$, namely $\{\text{components of } \ms\cap B\st B\in\mathcal{B}^{(i)}\}$ for $i=1,\ldots, C_{BC}$ and $\{\text{components of }\ms\cap T \st T\in\mathcal{T}\}$.

Analogously, we have that
\[
\mathcal{U}_{\partial\ms}\eqdef \bigcup_{i=1}^{C_{BC}} \bigcup_{B\in\mathcal{B}^{(i)}} \{\text{components of }\partial\ms\cap B \}\cup \bigcup_{T\in\mathcal{T}} \{\text{components of }\partial \ms\cap T\}
\]
is a controlled cover of $\partial\ms$  with overlap at most $C_{BC}+1$ and sum of the Betti numbers of any finite intersection of elements bounded above by $\alpha$.
Hence, by \cref{lem:cover-and-betti}, we get that
\begin{align*}
\sum_{k=0}^n b_k(\ms) +\sum_{k=0}^{n-1} b_k(\partial\ms) &\le C_{AC} (\abs{\mathcal{U}_\ms} + \abs{\mathcal{U}_{\partial\ms}}) .
\end{align*}
Now observe that, by \cref{prop:write-as-graph}\ref{wag:ii} and \cref{prop:sheeting-results}, we have
\[
\abs{\{\text{components of }\partial\ms\cap B \}} \le \abs{\{\text{components of }\ms\cap B \}} \le C_{CS}\CA.
\]
Moreover, $\abs{\{\text{components of } \ms\cap T\}}$ and $\abs{\{\text{components of }\partial \ms\cap T\}}$ are bounded by $C_{CT}\CA$ by \cref{prop:ConicalRegion}\ref{cr:sheeting-telescopes}, for all $T\in \mathcal{T}$. Finally, \cref{prop:ConicalRegion}\ref{cr:controlled-overlap} implies that
\[
\abs{\mathcal{T}} \le 2 \sum_{i=1}^{C_{BC}} \abs{\mathcal{B}^{(i)}}.
\]
Therefore, we get
\begin{align*}
\abs{\mathcal{U}_\ms} + \abs{\mathcal{U}_{\partial\ms}} &\le  2C_{CS}\CA\sum_{i=1}^{C_{BC}} \abs{\mathcal{B}^{(i)}} +   2C_{CT}\CA\abs{\mathcal{T}} \le (2C_{CS} + 4 C_{CT})\CA \sum_{i=1}^{C_{BC}} \abs{\mathcal{B}^{(i)}} \\
&\le (2C_{CS} + 4 C_{CT})\CA C_{BC} \abs{\mathcal{B}^{(1)}} \le (2C_{CS} + 4 C_{CT}) C_{BC} C_{MI} \CA (\ind(\ms)+1),
\end{align*}
where we used \cref{thm:DisjointIndex} in the last inequality.
Therefore, by putting together the estimates, we get
\[
\sum_{k=0}^n b_k(\ms) +\sum_{k=0}^{n-1} b_k(\partial\ms) \le C_\CA (1+\ind(\ms)),
\]
where
\begin{equation}\label{eq:dep-const}
C_\CA\eqdef C_{AC} (2C_{CS} + 4 C_{CT}) C_{BC} C_{MI} \CA.
\end{equation}
Observe that, in dimension $n+1=3$, all the constants $C_{AC}$, $C_{CS}$ $C_{CT}$, $C_{BC}$, $C_{MI}$ depend only on {$(M,g)$}. Therefore, we can write $C_\CA=C\CA$ for some constant $C=C(M,g)>0$ independent of $\CA$ and the proof is concluded.

\bibliography{biblio}

\printaddress
\end{document}